\documentclass[11pt]{amsart}
\usepackage{a4,amsmath,amssymb,amsthm,eucal}
\usepackage{caption,subfig}
\usepackage{epsfig}
\usepackage{tikz}
\newcommand\N{\mathbb{N}}
\newcommand\R{\mathbb{R}}

\newcommand{\eps}{\varepsilon}
\newcommand{\deb}{\rightharpoonup}
\newcommand{\HH}{{\mathcal H}}
\renewcommand{\H}{\HH^1}

\newcommand{\forget}[1]{}
\def\Om{{\Omega}} 
\def\om2{{\Om\times\Om}}

\def\supp{\mathrm{supp}\,}

\def\dist{\mathrm{dist}\,}
\def\clco{\overline{\mathrm{co}}\,}
\def\diam{\mathrm{diam}\,}

\def\dist{\mathrm{dist}\,}
\def\eps{\varepsilon}

\newcommand{\res}{\llcorner} 

\newtheorem{theorem}{Theorem}[section]
\newtheorem{proposition}[theorem]{Proposition}

\newtheorem{definition}[theorem]{Definition}

\newtheorem{lemma}[theorem]{Lemma}
\newtheorem{corollary}[theorem]{Corollary}

\newtheorem{problem}[theorem]{Problem}

\theoremstyle{remark}
\newtheorem{remark}[theorem]{Remark}
\newtheorem{example}[theorem]{Example}

\title[Asymptotic optimal location of facilities]{Asymptotic optimal location of facilities in a competition between population and industries}
\author{G.~Buttazzo \and F.~Santambrogio \and E.~Stepanov}
\address[G. Buttazzo]{{\tt buttazzo@dm.unipi.it}, Dipartimento di Matematica, Universit\`{a} di Pisa,
Largo B.~Pontecorvo~5, 56127 Pisa, Italy}
\address[F. Santambrogio]{{\tt filippo.santambrogio@math.u-psud.fr}, Laboratoire de Math\'ematiques d'Orsay, Universit\'{e} Paris-Sud 11, 91405 Orsay cedex, France}
\address[E. Stepanov]{{\tt stepanov.eugene@gmail.com}, Dipartimento di Matematica, Universit\`{a} di Pisa,
Largo B.~Pontecorvo~5, 56127 Pisa, Italy \and
Department of Mathematical Physics, Faculty of Mathematics and Mechanics,
St. Petersburg State University, Universitetskij pr.~28, Old Peterhof,
198504 St.Petersburg, Russia}
\thanks{The support of the projects EVaMEF ANR-09-JCJC-0096-01 and ANR-07-BLAN-0235 OTARIE is acknowledged. The work of the third author was also financed also by GNAMPA and by RFBR grant \#11-01-00825.
The work of the first and the third is part of the project 2008K7Z249 ``Trasporto ottimo di massa,
disuguaglianze geometriche e funzionali e applicazioni'' financed the
Italian Ministry of Research.}
\date{February 17, 2011}

\begin{document}
\begin{abstract}
We consider the problem of optimally locating a given number $k$ of points in
$\R^n$ for an integral cost function which takes into account two measures
$\varphi^+$ and $\varphi^-$. The points represent for example new
industrial facilities that have to be located, the measure $\varphi^+$ representing in this case already
existing industries that want to be close to the new ones, and $\varphi^-$
representing
private citizens who want to stay far away. The asymptotic analysis as
$k\to\infty$ is performed, providing the asymptotic density of optimal
locations.
\end{abstract}
\maketitle

\section{Introduction}\label{sec_loc1Intro}


A typical problem in facility location can be mathematically described through the choice of a given number of points in a domain so as to minimize an ``average distance'' criterion, the average being computed with respect to a measure $\varphi$. More precisely, for every subset $\Sigma\subset \R^n$ define
\[
F(\Sigma):=\int_{\R^n} \dist(x,\Sigma)\, d\varphi(x),
\]
where $\dist(x, \Sigma):=\inf_{y\in \Sigma} d(x,y)$ is the distance between $x$ and $\Sigma$. In this paper we study the following problem.

\begin{problem}\label{pb_loc1loc}
Find a $\Sigma=\Sigma_{opt}\subset \R^n$ minimizing the functional $F$ among all sets $\Sigma\subset \R^n$ satisfying $\#\Sigma\leq k$. In other words, denoting by ${\mathcal A}_k$ the set of admissible $\Sigma$, i.e.
\[
{\mathcal A}_k:=\{\Sigma\subset \R^n\,:\, \#\Sigma\leq k\},
\]
we are interesting in finding
\[
\min\{F(\Sigma)\,:\, \Sigma\in {\mathcal A}_k\}.
\]
\end{problem}

This problem has been intensively studied when $\varphi$ is a positive measure with finite mass and compact support. Here in the paper we want to analyze what happens when the positivity assumption is dropped, thus taking $\varphi=\varphi^+-\varphi^-$.

One can easily give the above problem (even for signed measures $\varphi$) an obvious economic interpretation useful especially for urban planning. Namely, we suppose that the support of $\varphi$ stands for some populated area (say, a city). Problem~\ref{pb_loc1loc} may be viewed as a simplified model of finding the optimal location $\Sigma$ of at most $k\in\N$ identical new industrial facilities (e.g.\ plants) given the distribution $\varphi^-$ of the population and that of the existing industries $\varphi^+$, both weighted with their respective influence, so that more influential industries or populated areas count more (although this is certainly not the only possible interpretation, e.g.\ one can think of coal-burning electric generating plants that have to be placed close to coal mines and far from the population). The cost function $F$ has then quite a clear meaning.
In fact, the integrals $\int_{\R^n}\dist(x,\Sigma)\,d\varphi^\pm$ measure how close in average the new facilities are to the population and to the existing industries (we call them for simplicity average transportation costs, although such a meaning can be more naturally attributed only to the integral with respect to $\varphi^+$). It has to be noted that, usually, people like to stay away as far as possible from new industrial facilities (because they are polluting, noisy, or spoiling the view from their windows) and thus are interested in increasing $\int_{\R^n}\dist(x,\Sigma)\,d\varphi^-$, while existing industries and the new ones in general are interested in staying as much as possible close to each other (at least, to minimize transportation costs for the new production), hence are inclined to minimize $\int_{\R^n} \dist(x,\Sigma)\, d\varphi^+$. The total cost $F$ takes into account both.  The natural question we investigate in the present paper reads then as follows:
\begin{itemize}
\item[] What is the asymptotic behavior of minimizers and minimum values of Problem~\ref{pb_loc1loc} as $k\to \infty$?
\end{itemize}
When answering to this question we are in particular obliged to study who wins in the ``competition'' between the population and the industries, namely, is the population $\varphi^-$ able to push the new facilities too far from  the existing industries $\varphi^+$, and are the existing industries $\varphi^+$ able to push the new facilities ``to the doors'' of private homes (i.e. too close to $\varphi^-$).

We refer to the above problem as the \emph{Fermat-Weber} or {\em optimal location} problem. It is usually studied for $\varphi^-=0$, in which case it is often referred to also as the {\em $k$-median} (or {\em multimedian}, or \emph{location-allocation}) problem. The economic interpretation is then that of finding the optimal location $\Sigma$ of $k$ identical facilities (e.g., shops, distribution centers etc.), and this is exactly the spirit in which this problem has been introduced by the German economist A.~Weber in~\cite{Weber1909}, though its applications go far beyond urban planning and economics and range from probability and statistics~\cite{GrafLuschgy00} to control theory~\cite{Liberzon03} (see e.g.~\cite{SuzDrez96,SuzOk95,MorgBolt01} and especially~\cite{GrafLuschgy00} for recent surveys on the subject). It is also worth remarking that the name of Fermat appears in this context because when $\varphi^+$ is given by three Dirac masses, then this problem becomes the famous problem of finding a point in the triangle minimizing the sum of distances to the vertices, posed by Fermat and then solved by Torricelli.

The vast majority of papers dealing with the classical location-allocation problem (i.e.\ with $\varphi^-=0$) consider only the discrete case, namely, when $\varphi^+$ is a sum of a finite number of Dirac masses. The continuous case (of not necessarily discrete measures $\varphi^+$) is dealt with relatively more rarely, though one should mention~\cite{GrafLuschgy00,MorgBolt01,FekMitBeur05} (see also references therein) that primarily treat this situation. In this continuous framework, the asymptotic behavior of minimizers to such a problem has received a lot of attention, since it is a question that only arises when one leaves the discrete case. Again, we refer to~\cite{GrafLuschgy00} for the more or less complete survey, but we also mention the recent papers~\cite{BouchJimMah02,TillMosc02} which obtain the results similar to those of~\cite{GrafLuschgy00} on the asymptotic behavior of minimizers using the $\Gamma$-convergence theory, as well as~\cite{BouchJimMah10} which studies from the point of view of $\Gamma$-convergence a very general class of asymptotic facility location problems. For the sake of completeness of the overview, we mention also some related results on the asymptotical analysis of random positioning of points (see, e.g.~\cite{Cohort04}), as well as on the dynamical location-allocation~\cite{BraButSamSte09}.

In this paper we mainly study the above question characterizing the limiting behavior of minimizers.  We first identify the limit of the minimal values of  Problem~\ref{pb_loc1loc}, which converge to
\begin{equation}\label{eunpb}
\min\{F(M)\,:\, M\subset \R^n \mbox{ closed}\},
\end{equation}
i.e.\ the cardinality constraint disappears as $k\to \infty$. Notice that this problem is non-trivial only in the case where $\varphi^-\neq 0$, since otherwise the obvious solution is $M:=\R^n$. Then we guess that the difference between the minimal value in Problem~\ref{pb_loc1loc} and in the unconstrained problem~\eqref{eunpb} is of the order of $k^{-1/n}$, as in the other asymptotical location results, and we prove the respective $\Gamma-$convergence result (Theorem~\ref{th_loc1Gamloc}) after this rescaling. From this convergence result we infer the limit behavior of the minimizers: not only they converge in the Hausdorff sense to a closed set minimizing the unconstrained problem~\eqref{eunpb}, but we also find convergence results for the density of the points of the optimal sets $\Sigma_k$, in the same spirit as it has been done in~\cite{BouchJimMah02,TillMosc02,BouchJimMah10} for the case $\varphi^-=0$.

Few words have to be said about possible generalizations and extensions of
our problem setting. First, instead of considering the transportation cost to be equal to the (Euclidean) distance, one could have considered some nondecreasing functions of a distance (usually one takes power functions), possibly different for the part of the functional depending on the measure $\varphi^+$ and that depending on $\varphi^-$. In this case one expects similar results up to a different rescaling of the functional (in the case $\varphi^-=0$ this is done in the above cited references). Further, instead of making a constraint on the number of points one could also study the penalizations depending on the cardinality of the set. In this paper we deliberately sacrifice such extensions for the sake of simplicity of the presentation of the technique and of the clarity of the result, since the respective extensions can be made relatively easily following the same order of ideas.

The paper is organized as follows. Section~\ref{sec_loc1not} gathers the necessary notation, while Section~\ref{sec_loc1exist} proves that Problem~\ref{pb_loc1loc} admits a solution. Section~\ref{sec_loc1lim0} considers the minimization problem without the cardinality constraint, i.e.\ Problem~\eqref{eunpb}. From Section~\ref{sec_loc1limdens} on, we want to consider the limit ``density'' of the optimal sets $\Sigma_k$ (i.e.\ the ``average number of points per unit volume''): this is done by means of $\Gamma-$convergence, a tool which is introduced in~\cite{DeGFra} to deal with limits of minimization problems. We will recall the fundamental definitions and introduce our $\Gamma-$convergence statement in Section~\ref{sec_loc1limdens}, and prove the results in Section~\ref{sec_loc1limdens1}.
In the Appendix we collect some results on sets satisfying the uniform external ball condition which are used in the paper since we will prove that optimal sets $M$ for~\eqref{eunpb} satisfy such a property, but these results are also of some independent interest.

\section{Notation}\label{sec_loc1not}

The Euclidean norm in $\R^n$ is denoted by $|\cdot|$, and the Euclidean distance between two points $x$ and $y$ by $d(x,y)$. The notation $B_r(x)\subset \R^n$ will always stand for the open ball of radius $r>0$ with center $x\in \R^n$. By $d_H$ we denote the Hausdorff distance between sets.
For a set $E\subset \R^n$ we denote by $1_E$ its characteristic function, by $E^c$ its complement, by $\bar E$ its closure,
by $\partial E$ its boundary,
by $\diam E$ its diameter and, for given $\varepsilon>0$, by $(E)_\varepsilon$ its $\varepsilon$-enlargement defined by
\[
(E)_\varepsilon:=\bigcup_{x\in E} B_\varepsilon(x).
\]

We denote by ${\mathcal L}^n$ the Lebesgue $n$-dimensional measure and by $\HH^k$ the $k$-dimen\-sional Hausdorff measure. All the other measures considered in this paper will be silently assumed
to be signed Borel measures with finite total variation and compact support in $\R^n$ if not otherwise explicitly stated. The support of a measure $\varphi$ is denoted by $\supp\varphi$.

For a closed set $M\subset \R^n$ and for an $x\in \R^n$ we denote by
$\pi_M(x)$ the projection of $x$ to $M$, i.e.\ the point of minimum distance from $x$ to $M$, if such a point is unique. This map is defined everywhere outside of a set $\mathfrak{R}_M$ called \emph{ridge set} of $M$. It is known that $\mathfrak{R}_M$ has zero $n$-dimensional Lebesgue measure since it is the set where the Lipschitz function $\dist(\cdot,M)$ is not differentiable. Moreover, the latter set is also known to be $(\HH^{n-1},n-1)$ rectifiable (see, e.g., Proposition~3.9 from~\cite{MantMen03} where even a slightly stronger result is proven and in a more general context of Riemannian manifold instead of $\R^n$).

As usual, the notation $L^p(\Omega)$ for an open subset $\Omega\subset\R^n$ stands for the respective Lebesgue space. The norm in this space is denoted by $\|\cdot\|_p$. The space $BV(\Omega)$ stands for the space of functions of bounded variation over $\Omega$ (i.e. such that their distributional derivatives are finite measures).

\section{Existence of solutions}\label{sec_loc1exist}

In order to rule out any doubt about the fact that the problems we investigate are well-posed, let us prove first the existence of solutions to Problem~\ref{pb_loc1loc}.

\begin{theorem}\label{th_loc1exist0}
Let $\varphi^\pm$ be finite positive Borel measures with compact supports in $\R^n$ and with $\varphi^+(\R^n)>\varphi^-(\R^n)$. Then Problem~\ref{pb_loc1loc} admits a solution. Furthermore, there is a ball $B$ such that for each $k$ there is a solution $\Sigma_k$ to Problem~\ref{pb_loc1loc} contained in $B$.
\end{theorem}

Before proving the above Theorem~\ref{th_loc1exist0} we show
that the strict inequality 
\[
\varphi^+(\R^n)>\varphi^-(\R^n)
\]
is essential for the existence of a solution (otherwise there may be no solution to Problem~\ref{pb_loc1loc}, even if $k=1$).

\begin{example}\label{ex_loc1_nonexist1}
Let $\varphi^+$ be the uniform probability measure over the unit circumference $\partial B_1(0)\subset\R^2$, i.e. $\varphi^-:=\frac{1}{2\pi}\H\res\partial B_1(0)$, and $\varphi^-$ be the Dirac mass concentrated in the origin. Then Problem~\ref{pb_loc1loc} with $k:=1$ admits no solution. In fact, for every point $z\in \R^2$, denoting by $r:=|z|$ and
\begin{align*}
f(r):= F(\{z\})& =-r+\frac{1}{2\pi}\int_{\partial B_1(0)}|x-z|\,d\H(x)\\
& =
-r+\frac{1}{2\pi}\int_{\partial B_1(0)}|x-(r,0)|\,d\H(x),
\end{align*}
we get for the derivative of the above function
\[
f'(r)=-1+\frac{1}{2\pi}\int_{B_1(0)}\frac{(x-(r,0))}{|x-(r,0)|}\,d\H(x)<0.
\]
Note that in this case we have $\varphi^+(\R^2)=\varphi^-(\R^2)$.
\end{example}

To prove Theorem~\ref{th_loc1exist0} we first introduce the following notation. For a closed set $M\subset \R^n$ let $\mathrm{Ess}\, M\subset M$ stand for the set of such points
$x\in M$ for which there is an $y\in \supp\varphi$ (possibly depending on $x$) such that
\[
d(y,x)=\dist(y,M).
\]
One clearly has then
\begin{equation}\label{eq_loc1_ess1}
F(\mathrm{Ess}\, M)=F(M),
\end{equation}
that is, $\mathrm{Ess}\, M$ is the ``essential'' part of $M$ (the points outside of which do not count for the value of the functional), and this justifies our notation. It is also immediate to notice that $\mathrm{Ess}\, M$ is closed whenever so is $M$. This is due to the fact that the support of $\varphi$ is compact; otherwise it is not true as seen for instance in the example of a closed interval $M:=\{0\}\times [-\pi/2, \pi/2]$ and a $\varphi$ with
$\supp\varphi$ being the graph of the function $y=\arctan x$, in which case
 $\mathrm{Ess}\, M$
is an open interval
 $\mathrm{Ess}\, M=\{0\}\times (-\pi/2, \pi/2)$.

We also need the following lemma.

\begin{lemma}\label{lm_F_lim}
Let $\Sigma_j$ be an arbitrary sequence of closed sets such that $F(\Sigma_j)$ is bounded from above. Then, under the 
assumption $\varphi^+(\R^n)>\varphi^-(\R^n)$, there exists a ball $B$ such that $\mathrm{Ess}\,\Sigma_j\subset B$ for every $j$.
\end{lemma}

\begin{proof}
Let us fix a ball $B_R(0)$ containing the support of $\varphi$, which is supposed to be compact. If the assertion is false, then there is a sequence $x_j\in\mathrm{Ess}\, \Sigma_{k_j}$ with $x_j\to\infty$ as $j\to\infty$. This, of course, implies that for every $y_j\in\supp\varphi$ such that $d(y_j,x_j)=\dist(y_j,\Sigma_{k_j})$ one has
\[
\dist(y_j,\Sigma_{k_j}) = d(y_j,x_j)\geq |x_j|-|y_j|\geq |x_j|-R\to \infty,
\]
as $j\to \infty$. Taking into account, for every $x_j'\in \Sigma_{k_j}$, the inequalities
\[
|y_j|+|x_j'|\geq d(y_j, x_j')\geq \dist(y_j,\Sigma_{k_j}),
\]
we get $x_j'\to \infty$. Let  $R_j:=\min\{|x|\,:\,x\in\Sigma_{k_j}\}$ and apply this last inequality to the points $x'_j\in \Sigma_{k_j}$ such that $|x'_j|=R_j$. We get $R_j\to\infty$ and, since
$$F(\Sigma_{k_j})\ge(R_j-R)\varphi^+(\R^n)-(R_j+R)\varphi^-(\R^n)=R_j(\varphi^+(\R^n)-\varphi^-(\R^n))-C$$
we also get $F(\Sigma_{k_j})\to\infty$ (due to the assumption $\varphi^+(\R^n)>\varphi^-(\R^n))$, which is a contradiction with the boundedness of $F(\Sigma_{k_j})$.
\end{proof}

Now we prove Theorem \ref{th_loc1exist0}.

\begin{proof}[Proof of Theorem  \ref{th_loc1exist0}]
To prove existence, for every $k$, of a minimizer in ${\mathcal A}_k$, we just apply the previous Lemma \ref{lm_F_lim} to any minimizing sequence $\Sigma_j\in\mathcal{A}_k$. Without loss of generality we may assume $\Sigma_j=\mathrm{Ess}\,\Sigma_j$ (otherwise just replace every set with its essential part). This provides uniform boundedness for such sets $\Sigma_j$. We are hence minimizing a continuous function over a compact subset of $(\R^{n})^k$, and the existence of a minimizer is straightforward.

Consider now a sequence of minimizers $\Sigma_k\in{\mathcal A}_k$. Notice that, by minimality, since ${\mathcal A}_1\subset {\mathcal A}_k$, we have $F(\Sigma_k)\leq F(\Sigma_1)$. This allows to apply again Lemma \ref{lm_F_lim} and prove that any sequence of essential minimizers is contained in the same ball, thus getting the second part of the statement.
\end{proof}

\section{Limit set}\label{sec_loc1lim0}

In this section we consider the problem
\begin{equation}\label{eq_loc1lim0P}
\min\{F(M)\,:\, M\subset \R^n \mbox{ closed}\}.
\end{equation}

\begin{proposition}\label{prop_loc1existLim0}
Let $\varphi^+(\R^n)>\varphi^-(\R^n)$. Then Problem~\eqref{eq_loc1lim0P} admits a minimizer, that can be taken compact.
\end{proposition}

\begin{proof}
Again, as in the proof of Theorem~\ref{th_loc1exist0} let $M_j$ be a minimizing sequence of closed sets (without any additional constraints) for $F$.  Without loss of generality we can assume that they are ``essential'' (i.e.\ $\mbox{Ess}\, M_j=M_j$), otherwise, take the essential parts of the latter, observing that the essential part of a closed set is still closed. Lemma \ref{lm_F_lim} gives the existence of a sufficiently large ball $B\subset\R^n$ (which without loss of generality will be assumed closed) such that $M_j\subset B$ for all sufficiently large $j\in\N$. According to the Blaschke theorem (Theorem~4.4.6 of~\cite{AmbrTilli00}) one has $M_j\to M\subset B$ in the sense of Hausdorff convergence up to a subsequence (not relabeled), and keeping in mind  the continuity of $F$ with respect to this convergence, we obtain that $M$ is a minimizer of~\eqref{eq_loc1lim0P} (which in particular, is compact).
\end{proof}

We notice now the following easy but important property of minimizers
to Problem~\eqref{eq_loc1lim0P}.

\begin{proposition}\label{prop_loc1chargeLim0}
Let $\varphi^+(\R^n)>\varphi^-(\R^n)$, and let $M$ be any minimizer of Problem~\eqref{eq_loc1lim0P}. Then $\varphi^+(M)\geq \varphi^+(\R^n)-\varphi^-(\R^n)>0$.
\end{proposition}

\begin{proof}
To prove that $\varphi^+(M)>0$, note that
for every $\eps >0$ one has
\begin{align*}
\dist(x, (M)_\eps)& \leq \dist(x, M) -\eps, \qquad x\not \in (M)_\eps,
\end{align*}
while for all $x\in \R^n$ one has
\begin{align*}
\dist(x, (M)_\eps)& \leq \dist(x, M),\\
\dist(x, (M)_\eps)& \geq \dist(x, M) -\eps.
\end{align*}
Hence,
we get
\begin{align*}
    \int_{\R^n}\dist(x, (M)_\eps)\, d\varphi^-(x)& \geq \int_{\R^n}\dist(x, M)\, d\varphi^-(x)-\eps\varphi^-(\R^n),\\
    \int_{\R^n}\dist(x, (M)_\eps)\, d\varphi^+(x)&=\int_{\R^n\setminus (M)_\eps}\dist(x, (M)_\eps)\, d\varphi^+(x)+\\
     & \quad \int_{(M)_\eps}\dist(x, (M)_\eps)\, d\varphi^+(x)\\
     & \leq \int_{\R^n}\dist(x, M)\, d\varphi^+(x)-\eps\varphi^+(\R^n\setminus (M)_\eps).
\end{align*}
Therefore,
\begin{equation}\label{eq_loc1lim0est1}
\begin{aligned}
F( (M)_\eps)&\leq F(M)+ \eps\left(\varphi^-(\R^n)-\varphi^+(\R^n\setminus (M)_\eps)\right)\\
& = F(M)+ \eps\left(\varphi^-(\R^n)-\varphi^+(\R^n)\right) +\eps \varphi^+((M)_\eps).
\end{aligned}
\end{equation}
Keeping in mind  that $\varphi^+((M)_\eps)\to \varphi^+(M)$ as $\eps\to 0^+$, we get that
\[
\varphi^+(M)\geq \varphi^+(\R^n)-\varphi^-(\R^n)>0,
\]
since otherwise the estimate~\eqref{eq_loc1lim0est1} together with the assumption
$\varphi^-(\R^n)<\varphi^+(\R^n)$ would give $F( (M)_\eps) <F(M)$ for sufficiently small $\eps>0$, contrary to the optimality of $M$.
%
\end{proof}

It is important to note that in general Problem~\eqref{eq_loc1lim0P} admits
many minimizers, both compact and noncompact (see Example~\ref{ex_loc1_Mopt1} below).
In the following statement we propose to select a particular minimizer (which will be always unbounded), that will play a special role in what follows.

\begin{proposition}\label{prop_loc1Mprime}
If $\Sigma$ is a minimizer of Problem~\eqref{eq_loc1lim0P}, then the closed set
\begin{equation}\label{eq_loc1defM1}
M:=\bigcap_{y\in \supp \varphi^-} B^c_{\dist(y, \Sigma)} (y).
\end{equation}
still solves the same problem, while
$M$ contains $\Sigma$
and
\begin{align*}
    \int_{\R^n}\dist(x, M)\, d\varphi^+(x)& \leq \int_{\R^n}\dist(x, \Sigma)\, d\varphi^+(x),\\
    \int_{\R^n}\dist(x, M)\, d\varphi^-(x)& = \int_{\R^n}\dist(x, \Sigma)\, d\varphi^-(x).
\end{align*}
In particular, $\varphi^+(M\setminus \Sigma)=0$.
\end{proposition}

\begin{proof}
Clearly, $M$ is closed.
We also note that $\Sigma\subset M$.
In fact, otherwise, there is an $x\in \Sigma$ such that $x\not\in M$, i.e.\ $x\in B_{\dist(y, \Sigma)} (y)$ for
some $y\in \supp\varphi^-$, or, in other words, $|x-y|<\dist(y, \Sigma)$ which is absurd.
Therefore,
\begin{align*}
    \int_{\R^n}\dist(x, M)\, d\varphi^+(x)& \leq \int_{\R^n}\dist(x, \Sigma)\, d\varphi^+(x).
\end{align*}
On the other hand, by construction of $M$ one has for every $y\in \supp\varphi^-$ that
\begin{equation}\label{eq_loc1dist1}
\dist(y, M)=\dist(y,\Sigma).
\end{equation}
In fact, $\dist(y, M)\leq \dist(y,\Sigma)$ for all $y\in \R^n$ since $\Sigma\subset M$, while
for every $y\in \supp\varphi^-$ and for every $x\in M$ one has $|y-x|\geq \dist(y,\Sigma)$, hence
\[
\dist(y, M)\geq \dist(y,\Sigma).
\]
The equality~\eqref{eq_loc1dist1} implies then
\begin{align*}
    \int_{\R^n}\dist(x, M)\, d\varphi^-(x)& = \int_{\R^n}\dist(x, \Sigma)\, d\varphi^-(x).
\end{align*}
Hence,
\[
F(M)\leq F(\Sigma),
\]
that is, $M$ is a minimizer of~\eqref{eq_loc1lim0P}. The last assertion is true since otherwise the first inequality becomes
strict contradicting the optimality of $\Sigma$.
\end{proof}

From now on we will call every minimizer $M$ of Problem~\eqref{eq_loc1lim0P} satisfying~\eqref{eq_loc1defM1}, where $\Sigma$ is some minimizer of the same problem, \emph{canonical} with respect to $\Sigma$ or simply canonical (if the reference to $\Sigma$ is unnecessary).

\begin{example}\label{ex_loc1_Mopt1}
Let $\varphi^-:= \delta_0$ be a Dirac mass concentrated in the origin,
and $\varphi^+\ll {\mathcal L}^n$ be such that
\[
\varphi^+(\R^n) >1=\varphi^-(\R^n).
\]
Then every canonical minimizer $M$ of Problem~\eqref{eq_loc1lim0P} is the complement of an open ball $B_r(0)$. To find it, we consider the function
\[
f(r):=F(B_r^c(0))=\int_{\R^n}(r-|x|)^+\,d\varphi^+(x)- r,
\]
so that finding a canonical minimizer amounts to minimizing $f$. One easily gets for the derivative of $f$ the expression
\[
f'(r)=\varphi^+(B_r(0))-1,
\]
which gives for the minimum (where $f'(r)=0$) the expression
\[
\varphi^+(B_r(0))=1.
\]
The latter determines uniquely the canonical minimizer.
Clearly however, the minimizers (not necessarily canonical) of Problem~\eqref{eq_loc1lim0P} are not unique. In fact, for instance also $B_r(0)^c\cap\supp\varphi^+$ is a minimizer.
\end{example}

It is worth remarking that although the canonical minimizer was unique in the above Example~\ref{ex_loc1_Mopt1}, we do not know whether this is true in general.

We now consider another important question, namely, when a minimizer
of Problem~\eqref{eq_loc1lim0P} is located a positive distance away from the support of $\varphi^-$.

\begin{proposition}\label{prop_loc1chargeLim1}
Let 
$M$ be any minimizer of Problem~\eqref{eq_loc1lim0P}.
If either
\[
\clco\supp \varphi^+\cap \clco\supp\varphi^-=\emptyset,
\]
where $\clco$ stands for the closed convex envelope of a set,
or
\[
\dist(\supp \varphi^+,\supp\varphi^-)>\diam\supp\varphi^+,
\]
then $M\cap \supp \varphi^-=\emptyset$.
\end{proposition}

\begin{proof}
We consider the two cases in two separate parts of the proof.

{\em Case 1}.
We consider first the case
\[
\clco\supp \varphi^+\cap \clco\supp\varphi^-=\emptyset.
\]
Then there is a hyperplane $\pi\subset \R^n$ such that
\[
\supp \varphi^\pm\subset \pi^\pm,
\]
where $\pi^+$ and $\pi^-$ stand for the open half-spaces bounded by $\pi$. We denote by $R\colon \R^n\to \R^n$ the reflection with respect to $\pi$, and set
\[
M^+:=M\cap\pi^+,\quad M^-:=M\cap\bar\pi^-,\quad\tilde M:=M^+\cup R(M^-).
\]
For every $x\in \pi^+$ (in particular, for $x\in\supp\varphi^+$) and $y\in\bar\pi^-$ (in particular, for $y\in M^-$) one has $|x-y|>|x-R(y)|$. Hence,
\[
\dist(x,M)\ge\dist(x,\tilde M),
\]
which implies
\begin{equation}\label{eq_loc1fp1}
\int_{\R^n}\dist(x,M)\,d\varphi^+(x)
\ge\int_{\R^n}\dist(x,\tilde M)\,d\varphi^+(x).
\end{equation}

One the other hand, consider any $x\in \pi^-$ : since for any $y\in \pi^-$ we have $|x-y|< |x-R(y)|$, we get on the contrary $\dist(x, M)\leq \dist(x, \tilde M)$. This implies
$$\int_{\R^n}\dist(x,M)\,d\varphi^-(x)
\le\int_{\R^n}\dist(x,\tilde M)\,d\varphi^-(x)$$
and, summing up, $F(M)\ge F(\tilde M)$.

Now, we argue by contradiction assuming that $M^-\ne\emptyset$. Take $x_0\in\supp\varphi^-\cap M\subset \pi^-$. For such a point $x_0$ one has $0=d(x_0,M)<d(x_0,\pi^+)\le d(x_0,\tilde M)$, which implies a strict inequality leading in the end to $F(M)>F(\tilde M)$ (since the same strict inequality will stay true in a neighborhood of $x_0$, which is charged by $\varphi^-$ since $x_0\in\supp\varphi^-$).

This gives a contradiction to the optimality of $M$.

{\em Case 2}.
We pass now to the case
\[
\dist(\supp \varphi^+,\supp\varphi^-)>\diam\supp\varphi^+.
\]
Let $\eps>0$ be such that
\[
\dist(\supp \varphi^+,\supp\varphi^-)>\diam\supp\varphi^+ +\eps.
\]
We claim that $M\cap (\supp\varphi^-)_\eps=\emptyset$ which would conclude the proof. In fact, otherwise for every $x\in M\cap (\supp\varphi^-)_\eps$ there is no $z\in \supp\varphi^+$ such that
\[
|z-x|=\dist(z,M),
\]
since this would mean that $M\cap \supp\varphi^+=\emptyset$ contrary to Proposition~\ref{prop_loc1chargeLim0}. Thus setting
\[
M':=M\setminus (\supp\varphi^-)_\eps,
\]
we get that $\dist(z, M')=\dist(z,M)$ for every $z\in \supp\varphi^+$, while
$\dist(z, M')\geq \dist(z,M)$ for every $z\in \supp\varphi^-$, and, moreover,
$\dist(z, M')> \dist(z,M)$ for a set of $z\in \supp\varphi^-$ of positive measure $\varphi^-$.
This would imply $F(M')>F(M)$ providing the desired contradiction with the optimality of $M$.
\end{proof}

We remark that for the above result to hold true, it is not enough to have
\[
\supp \varphi^+\cap \supp\varphi^-=\emptyset,
\]
as the following example shows.

\begin{example}\label{ex_loc1charge1Lim0}
Let $n=1$ and let
\[
\varphi^+:=b\delta_d+m\delta_{2R},\qquad
\varphi^-:=a\delta_0+c\delta_{R},
\]
with
\[
0<d< R/2,\quad 0<a<b(1-d/R),\quad c>a+b,\quad m>a+c.
\]
(see Figure~\ref{loc1_fig0}).

Hence $\dist(\supp \varphi^+,\supp\varphi^-)=d>0$. We show that $\{0, 2R\}$ is optimal and that for every optimal $\Sigma\subset \R$ one has
\[
\{0,2R\}\subset \Sigma,
\]
which implies in particular that
$\Sigma\cap \supp\varphi^-\neq\emptyset$.

To this aim, first note that $2R\in \Sigma$.  In fact, Proposition~\ref{prop_loc1chargeLim0} guarantees that $\varphi^+(\Sigma)>\varphi^+(\R)-\varphi^-(\R)$, but the mass of the point $d$ alone is not sufficient, because of the assumption $m>a+c$.

We have proved that $2R$ belongs to any optimal set $\Sigma$. Keep in mind that any optimal set $\Sigma$ may be replaced with $\Sigma':=\bigcup_{z\in\supp\varphi^+}\{x_z\}$, where $x_z\in\Sigma$ stands for an arbitrary point such that $|z-x_z|=\dist(z,\Sigma)$. The new set $\Sigma'\subset\Sigma$ is still optimal, since $\dist(z,\Sigma')=\dist(z,\Sigma)$ for every $z\in\supp\varphi^+$ and $\dist(z,\Sigma')=\dist(z,\Sigma)$ for every $z\in\supp\varphi^-$. In particular, in this case, this means that every optimal set must contain a smaller set composed of exactly two points, that is again optimal. And this optimal set must contain $2R$ as well. In practice, we are only lead to find the second point of this set, considering only sets of the form $\{x,2R\}$.

We are hence left with one only degree of freedom and we can consider the function $f(x):=F(\{x,2R\})$. Our goal is to prove that it is optimal at $x=0$.

This function is given by
\[
f(x)=b\left[|x-d|\wedge(2R-d)\right]-a\left[|x|\wedge(2R)\right]-c\left[|x-R|\wedge R\right].
\]

It is a piecewise linear function satisfying
\begin{align*}
&f'(x)=a-b &\mbox{ if } &2d-2R<x<0,\\
&f'(x)=-a-b+c &\mbox{ if } &0<x<d,
\end{align*}
and
\begin{align*}
f(0)&=bd-cR,\qquad f(R)=-aR+b(R-d),\\
f(2R)&=f(-2R)=-2aR-cR+b(2R-d).
\end{align*}
The point $0$ is the only minimizer of this function if and only if $f'<0$ at the left of $0$, $f'>0$ at the right of $0$, and at the other nodes one has the strict inequality $f(x)>f(0)$, which means that we impose
$$a-b<0,\quad -a-b+c>0,\quad f(R)\wedge f(2R)\wedge f(-2R)>f(0).$$
The assumptions guarantee $b>a$ and $c>a+b$; notice that
$$f(0)=bd-cR<bd-(a+b)R=f(R)-2b(R-d)<f(R).$$
 Moreover, the inequality $f(0)<f(2R)$ is exactly guaranteed by the assumption $a<b(1-d/R)$.
The conclusion is $0\in \Sigma$ as claimed.
\end{example}

\begin{center}
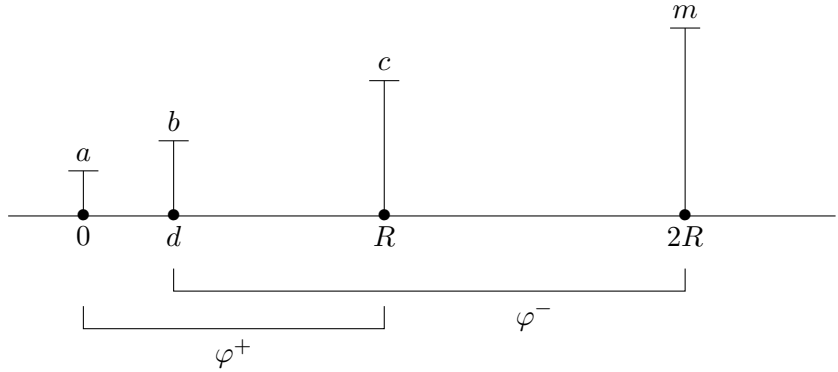
\begin{figure}
\begin{tikzpicture}

\draw (-1,0)  --  (10,0);
\draw (0,0) node[below] {$0$} node{$\bullet$};
\draw (1.2,0) node[below] {$d$}   node{$\bullet$};
\draw (4,0) node[below] {$R$}  node{$\bullet$};
\draw (8,0) node[below] {$2R$}  node{$\bullet$};

\draw (0,0)  --  (0,0.6);
\draw (-0.2,0.6) -- (0.2,0.6) node[midway,above] {$a$} ;

\draw (1.2,0)  --  (1.2,1);
\draw (1,1) -- (1.4,1) node[midway,above] {$b$} ;

\draw (4,0)  --  (4,1.8);
\draw (3.8,1.8) -- (4.2,1.8) node[midway,above] {$c$} ;

\draw (8,0)  --  (8,2.5);
\draw (7.8,2.5) -- (8.2,2.5) node[midway,above] {$m$} ;

\draw (0,-1.2) -- (0,-1.5)  --  (4,-1.5) node[midway,below] {$\varphi^+$} -- (4,-1.2);

\draw (1.2,-0.7) -- (1.2,-1)  --  (8,-1)  -- (8,-0.7);
\draw (6,-1) node[below] {$\varphi^-$};

\end{tikzpicture}

\caption{The measures $\varphi^\pm$ as in Example~\ref{ex_loc1charge1Lim0}: the vertical segments above the respective points stand for the respective masses.}
\label{loc1_fig0}
\end{figure}
\end{center}

\begin{proposition}\label{prop_loc1extballLim0}
Suppose $\varphi^+(\R^n)>\varphi^-(\R^n)$, let $\Sigma$ be any minimizer of~\eqref{eq_loc1lim0P}, satisfying $\Sigma\cap\supp\varphi^-=\emptyset$ (which is the case, for instance, if any of the conditions of Proposition~\ref{prop_loc1chargeLim1} hold)
and  let $M$ be given by~\eqref{eq_loc1defM1}.  Then $M$
satisfies the uniform external ball condition (see Definition~\ref{def_loc1_uball}), and, in particular,
$\partial M$ is $(\HH^{n-1},n-1)$-rectifiable.
\end{proposition}

\begin{proof}
Take a point $x\in\partial M$, and (by definition of boundary point), a sequence $x_k\to x$ with $x_k\notin M$. Then we have, by definition of $M$, $x_k\in B_{r_k}(y_k)$, with $y_k\in\supp\varphi^-$ and $r_k=\dist(y_k,\Sigma)$. Assuming, up to a subsequence (not relabeled), that
$y_k\to y\in \supp\varphi^-$, and passing to the limit as $k\to \infty$, we get $x\in \bar B_r(y)$ with $r=\dist(y,\Sigma)$. Since the whole $M$ is contained in the complement of the open ball $B_r(y)$, one obtains, for every $r'<r$, the existence of a ball whose boundary touches $M$ exactly at $x$ (it is sufficient to center this ball on the segment connecting $x$ to $y$). This gives the external ball condition, which is uniform since $r=\dist(y,\Sigma)$ is bounded from below, thanks to the assumption on $\Sigma$, which guarantees that $\Sigma$ and $\supp\varphi^-$ are a positive distance apart.
%
%
\end{proof}

We further deduce a necessary condition for the optimality of $M$, which, though not used in the sequel, is however of some independent interest.

\begin{proposition}\label{prop_loc1massbalance0}
Let $\phi_\eps\colon \R^n\to \R^n$ be a one parameter group of diffeomorphisms satisfying
\begin{equation}\label{eq_loc1generator}
\phi_\eps(x)=x+\eps X(x)+o(\eps),
\end{equation}
as $\eps\to 0$, where $X\in C^\infty_0(\R^n;\R^n)$.
Let $\varphi:=\varphi^+-\varphi^-$ be a Borel measure such that $\varphi(E)=0$ whenever
$\HH^{n-1}(E) <\infty$. Then for
all $X\in C^\infty_0(\R^n;\R^n)$ one has
\begin{equation}\label{eq_loc1massbalmain0}
\begin{aligned}
\frac{\partial}{\partial\eps} F(\phi_\eps(M)){\Big|}_{\eps=0}&=\int_{\R^n\setminus M}
\left\langle
X(\pi_M(x)),\frac{\pi_M(x)-x}{|\pi_M(x)-x|}\right\rangle \,d\varphi \\
& = \int_{\R^n\setminus M} \left\langle X(\pi_M(x)),\nabla \dist
(x,M)\right\rangle \,d\varphi ,
\end{aligned}
\end{equation}
where $\pi_M\colon \R^n\to M$ stands for the projection onto $M$ (defined everywhere outside of the ridge set of $M$).
In particular, if $M$ is a minimizer of $F$, then
\begin{equation}\label{eq_loc1massbalmain}
\begin{aligned}\int_{\R^n\setminus M}
\left\langle X(\pi_M(x)),\frac{\pi_M(x)-x}{|\pi_M(x)-x|}\right\rangle \,d\varphi =0
\end{aligned}
\end{equation}
for all $X\in C^\infty_0(\R^n;\R^n)$.
\end{proposition}

\begin{proof}
We adopt the method of calculation of the derivative of the distance function
with respect to the variation of the set, used in~\cite[Lemma~4.5]{AmbrMant98}.

For $z:=\phi_\eps(\pi_M(x))$ one clearly has
\[
\dist(x,M) = |\pi_M(x)-x|,\qquad\qquad
\dist(x,M_\eps) \leq |z-x|.
\]
From~\eqref{eq_loc1generator} we get, for $\eps\to 0$,
\begin{align*}
|z-x|^2&= \left\langle\pi_M(x)-x+\eps
X(\pi_M(x)),\pi_M(x)-x+\eps
X(\pi_M(x))\right\rangle +o(\eps)\\
&=|\pi_M(x)-x|^2+2\left\langle\pi_M(x)-x,\eps X(\pi_M(x))\right\rangle+o(\eps)\\
&=|\pi_M(x)-x|^2\left(1+2\left\langle\frac{\pi_M(x)-x}{|\pi_M(x)-x|^2},\eps
X(\pi_M(x))\right\rangle+o(\eps)\right).
\end{align*}
Then
\begin{align*}\displaystyle
 \dist(x,M_\eps)-\dist(x,M)&\leq |z-x|-|\pi_M(x)-x|\\
&=\eps\left\langle\frac{\pi_M(x)-x}{|\pi_M(x)-x|},
X(\pi_M(x))\right\rangle +o(\eps),
\end{align*}
and we deduce
\begin{equation}\label{first}
    \limsup_{\eps\to 0}\frac{1}{\eps}(\dist(x,M_\eps)-\dist(x,M))\leq
     \left\langle\frac{\pi_M(x)-x}{|\pi_M(x)-x|}, X(\pi_M(x))\right\rangle.
\end{equation}
On the other hand, consider a sequence $\eps_\nu\to 0^+$ for
$\nu\to\infty$. The set of points $x\in \R^n$ for which both
$\pi_M(x)$ 
and $\pi_{M_{\eps_\nu}}(x)$ are singletons for any $\nu\in\N$ is of full measure $\varphi$ in $\R^n$ (the complement is a countable union of ridge sets $\mathfrak{R}_{M_\nu}$ and $\mathfrak{R}_{M}$ which are all $(\HH^{n-1},n-1)$-rectifiable, hence $\varphi$-negligible). For all such $x$, since $\phi_\eps$ is invertible for all sufficiently small $\eps$, let $\zeta_\nu:=\phi_{\eps_\nu}^{-1}(\pi_{M_{\eps_\nu}}(x))$, so that
\[
\dist(x,M_{\eps_\nu})=|\phi_{\eps_\nu}(\zeta_\nu)-x|,
\qquad\dist(x,M)\le|\zeta_\nu-x|.
\]
Again we have
\begin{align*}\displaystyle
 |\phi_{\eps_\nu}(\zeta_\nu)-x|-|\zeta_\nu-x|&=|\zeta_\nu -x|\left(
 \sqrt{1+2\left\langle\frac{\zeta_\nu-x}{|\zeta_\nu-x|^2},\eps_\nu
 X(\zeta_\nu)\right\rangle +o(\eps_\nu)}-1\right)\\
  &=\eps_\nu\left\langle\frac{\zeta_\nu-x}{|\zeta_\nu-x|},
 X(\zeta_\nu)\right\rangle + o(\eps_\nu).
\end{align*}
Therefore,
\begin{equation*}
\dist(x,M_{\eps_\nu})-\dist(x,M)\geq
\eps_\nu\left\langle\frac{\zeta_\nu-x}{|\zeta_\nu-x|},
 X(\zeta_\nu)\right\rangle + o(\eps_\nu).
\end{equation*}
Passing to the limit as $\nu\to \infty$, we get
\begin{equation}\label{second}
     \left\langle\frac{\pi_M(x)-x}{|\pi_M(x)-x|},X(\pi_M(x))\right\rangle\le\liminf_{\nu\to\infty}\frac{1}{\eps_\nu}
     \left(\dist(x,M_{\eps_\nu})-\dist(x,M)\right).
\end{equation}
Combining~\eqref{first} with~\eqref{second}, we get for $\varphi$-a.e.\
$x\in\R^n$,
\[
    \lim_{\nu\to\infty}\frac{1}{\eps_\nu}
    (\dist(x,M_{\eps_\nu})-\dist(x,M))=
    \left\langle\frac{\pi_M(x)-x}{|\pi_M(x)-x|}, X(\pi_M(x))\right\rangle,
\]
so that, by Lebesgue dominated convergence theorem,
\[
    \lim_{\nu\to\infty}\frac{1}{\eps_\nu}
    \int_\Om(\dist(x,M_{\eps_\nu})-\dist(x,M))\,d\varphi=
    \int_{\R^n\setminus M}\left\langle\frac{\pi_M(x)-x}{|\pi_M(x)-x|}, X(\pi_M(x))\right\rangle\,d\varphi.
\]
Since the sequence $\eps_\nu$ is arbitrary, one has
\begin{align*}
    \lim_{\eps\to 0^+}\frac{1}{\eps}
    \int_{\R^n\setminus M}(\dist(x,M_{\eps}) & -\dist(x,M))\,d\varphi=\\
&
    \int_{\R^n\setminus M}\left\langle\frac{\pi_M(x)-x}{|\pi_M(x)-x|}, X(\pi_M(x))\right\rangle\,d\varphi,
\end{align*}
which concludes the proof.
\end{proof}

\begin{corollary}\label{co_loc1massbalance0a}
Under the conditions of Proposition~\ref{prop_loc1massbalance0},
if $M$ is a minimizer of Problem~\eqref{eq_loc1lim0P} such that for $\varphi$-a.e.\ $z\in M$ the set $(\pi_M)^{-1}(z)=\{x\,:\,\pi_M(x)=z\}$ is contained in a line
(this is true, for instance, when $\partial M$ is $C^{1,1}$), then
\begin{equation}\label{eq_loc1massbalmain0a}
(\pi_M)_{\#}(\varphi^+\res M^c) = (\pi_M)_{\#}(\varphi^-\res M^c).
\end{equation}
In particular, in this case under any of the conditions of Proposition~\ref{prop_loc1chargeLim1} one has
\begin{equation}\label{eq_loc1massbalmain0b}
\varphi^+(M)=\varphi^+(\R^n)-\varphi^-(\R^n)>0,
\end{equation}
which improves the result of Proposition \ref{prop_loc1chargeLim0}.
\end{corollary}

\begin{proof}
Disintegrating~\eqref{eq_loc1massbalmain} with respect to the projection
$\pi_M$, we get
\[
\int_{\pi_M(\R^n\setminus M)}
\left\langle X(z),\nu(z)\right\rangle\,d(\pi_M)_{\#}(\varphi\res M^c)(z)=0,
\]
where $\nu(z)$ stands for the unit direction of a line containing $(\pi_M)^{-1}(z)$, which gives, since
$X$ is arbitrary, $(\pi_M)_{\#}(\varphi\res M^c)=0$, and hence proves the validity of~\eqref{eq_loc1massbalmain0a}.
The latter then implies for the situations when $\supp \varphi^-\cap M=\emptyset$ that
\begin{align*}
(\varphi^+(\R^n)-\varphi^+(M))-\varphi^-(\R^n) &=\varphi^+(\R^n\setminus M)-\varphi^-(\R^n)\\
& =(\varphi\res M^c)(\R^n)=(\pi_M)_{\#}(\varphi\res M^c)(\R^n)=0,
\end{align*}
which provides~\eqref{eq_loc1massbalmain0b}.
\end{proof}

\section{Limiting density}\label{sec_loc1limdens}

In this section we study the asymptotic behavior of solutions to Problem~\ref{pb_loc1loc} as $k\to\infty$. This will be achieved by means of a $\Gamma$-convergence technique.

For the theory of $\Gamma-$convergence, we refer to~\cite{introgammaconve}, but we recall the main notions that we need.

\begin{definition}
 Let $X$ be a metric space and $G_k \colon X\to\R\cup\{\infty\}$ be a sequence of functionals. We define the new functionals $G^-$ and $G^+$ over $X$ (called $\Gamma-\liminf$ and $\Gamma-\limsup$ of this sequence respectively) by
\begin{align*}
G^-(x)& :=\inf\{\liminf_{k\to\infty} G_k(x_k)\;:\;x_k\to x\}, \\ G^+(x) &:=\inf\{\limsup_{k\to\infty} G_k(x_k)\;:\;x_k\to x\}.
\end{align*}
Should $G^-$ and $G^+$ coincide, then we say that $G_k$ is $\Gamma-$converging to the common value $G=G^-=G^+$.
\end{definition}

Among the properties of $\Gamma-$convergence the following are of utmost importance for us:
\begin{itemize}
\item if there exists a compact set $K\subset X$ such that $\inf_X G_k=\inf_K G_k$ for any $k$, then $F$ attains its minimum and $\inf G_k\to \min G$;
\item if $(x_k)_k$ is a sequence of minimizers for  $G_k$ admitting a subsequence converging to $x$, then $x$ minimizes $G$;
\item if $G_k$ $\Gamma-$converge to $G$, then $G_k+H$ $\Gamma-$converge to $G+H$ for any continuous function $H\colon X\to\R\cup\{\infty\}$.
\end{itemize}
The latter property is only presented so as to show the interest in proving a $\Gamma-$convergence result rather than only studying the limit behavior of minima and minimizers, due to the stability properties of this notion of limit.

We now want to define a sequence of functionals on a given metric space so as to read our asymptotic problem in terms of the $\Gamma-$convergence.

Let $D:=\text{co}\,\supp\varphi$. To fulfill our program, it is convenient to consider the set ${\mathcal A}=\cup_{k\in\N}{\mathcal A}_k$ of all sets $\Sigma\subset D$ satisfying $\#\Sigma<\infty$ (i.e. consisting of finite points) to be immersed in the set ${\mathcal P}(D)$ of Borel probability measures over $D$. This can be done by assigning to each nonempty $\Sigma\in{\mathcal A}$ the measure $\mu_\Sigma\in{\mathcal P}(D)$ defined by
\[
\mu_\Sigma(e):=\frac{\#\Sigma\cap e}{\#\Sigma}
\]
for each Borel $e\subset D$. For every $\mu\in{\mathcal P}(D)$ we set now
\[
G_k(\mu):=\left\{
\begin{array}{ll}
\displaystyle
k^{1/n}\Big(F(\Sigma_k)-\inf_A F(A)\Big)&\mbox{if }\mu=\mu_\Sigma,\,\Sigma\in\mathcal{A},\, \#\Sigma=k,\\
+\infty&\mbox{otherwise}.
\end{array}
\right.
\]
Here and in the sequel by writing $\inf_A F(A)$ we assume the infimum to be taken over closed sets. Our aim is to study $\Gamma$-convergence of the sequence of functionals $G_k:{\mathcal P}(D)\to\bar{\R}:=\R\cup\{+\infty\}$ as $k\to\infty$. The goal of immerging all the problems in the set of probability measures is twofold: on the one hand, we need to select a common space for the Problem~\ref{pb_loc1loc} with different values of $k$; on the other hand we need to choose it well so as to guarantee both compactness and a good interpretation in terms of densities.

To this aim we start with some auxiliary notation. Define
\begin{equation}\label{eq_loc1defthet}
\theta_{n}:=\inf\left\{
\liminf_{k\to\infty}k^{1/n}\int_{[0,1]^n}\dist(x,\Sigma)\,dx
\ :\ \#\Sigma_k=k,\ \Sigma\subset [0,1]^n
\right\}.
\end{equation}

The exact values of $\theta_n$ are known only in few cases. In particular,
the computation is immediate for $n=1$, with $\theta_1=1/4$, while for $n=2$, it is known (see, e.g.~\cite{MorgBolt01}, or Theorem~8.15 from~\cite{GrafLuschgy00}) that
$$
\theta_2=\int_\sigma|x|\,dx=\frac{4+3\log3}{6\sqrt2 3^{3/4}}\sim0.377
$$
where $\sigma\subset\R^2$ stands for the regular hexagon of unit area centered in the origin. However quite fine estimates both from above and from below on $\theta_{n}$ are known and can be found either in Chapter~8 of \cite{GrafLuschgy00} or in~\cite{But08-revmasstransp}. For our purpose it is enough to remark that $\theta_n \in(0,+\infty)$ (e.g. by Proposition~8.3 from~\cite{GrafLuschgy00} one has~$\theta_n\ge
\omega_n^{-1/n}n/(1+n)$, where $\omega_n$ stands for the volume of a unit $n$-dimensional ball, while a rather precise estimate on $\theta_n$ from above can be found, say, in Theorem~8.5 from~\cite{GrafLuschgy00} and a more rough but more easily applicable estimate can be found in~\cite{But08-revmasstransp}).

Recall that the Radon-Nikodym theorem (Theorem~2.17 from~\cite{AmbrFuscoPall00})
implies the existence of a unique representation
$$
\mu=\rho{\mathcal L}^n + \mu_{sing}
$$
for every finite Borel measure $\mu$ over $\R^n$, where $\mu_{sing}$ is singular with respect to the Lebesgue measure ${\mathcal L}^n$, while $\rho\in L^1(\R^n)$ and, thanks to the Besicovitch derivation theorem,
\[
\rho(x)=\lim_{\delta\downarrow 0}\frac{\mu(Q_\delta(x))}{|Q_\delta(x)|}
\]
for ${\mathcal L}^n$-a.e. $x\in \R^n$, $Q_\delta(x)\subset \R^n$ standing
for the cube of sidelength $\delta$ centered at $x$.
For a Borel set $M\subset \R^n$ and a measure $\mu\in {\mathcal P}(D)$ define then
\[
\Phi(\mu,M):=\theta_{n} \int_M \frac{d\varphi^+}{\rho^{1/n}}.
\]
The quantity $\Phi(\mu,M)$ represents the optimal value of the asymptotic optimal location problems with $\varphi^+\res M$ instead of $\varphi^+$ and $\varphi^-:=0$.
In particular, the following lemma is a corollary of a general $\Gamma-$convergence result from~\cite{TillMosc02}.

\begin{lemma}\label{lm_TilliMoscGammaloc}
For any positive measure $\varphi^+$ supported on $M$, one has
\[
\liminf_k(\# \Sigma_k)^{1/n} \int_M\dist(x, \Sigma_k)\, d\varphi^+(x) \geq \Phi(\mu,M),
\]
whenever
\[
\mu_{\Sigma_k}\rightharpoonup \mu 
\]
in the $*$-weak sense of measures as $k\to \infty$.  Moreover, for each probability measure $\mu$ over $M$, there exists a sequence of sets $\Sigma_k$ with $\# \Sigma_k\to\infty$, such that
$\mu_{\Sigma_k}\rightharpoonup \mu$ and
\[
\limsup_k(\# \Sigma_k)^{1/n} \int_M\dist(x, \Sigma_k)\, d\varphi^+(x) \leq \Phi(\mu,M).
\]
\end{lemma}

At last, for every $\mu\in {\mathcal P}(D)$ define
\begin{align*}
G_\infty(\mu) & := \inf\left\{ \Phi(\mu, A)\,:\, \supp\mu\subset A, \, A \mbox{ is a minimizer to~\eqref{eq_loc1lim0P}}
\right\}.
\end{align*}
The following easy observation will be useful in the sequel.

\begin{lemma}\label{lm_loc1_canonFinf1}
One has
\begin{align*}
G_\infty(\mu) & = \inf\left\{ \Phi(\mu, A)\,:\, \supp\mu\subset A, \, A \mbox{ is a canonical minimizer to~\eqref{eq_loc1lim0P}}
\right\}.
\end{align*}
\end{lemma}

\begin{proof}
If $A'$ is a minimizer to~\eqref{eq_loc1lim0P} and $A$ is a canonical minimizer to~\eqref{eq_loc1lim0P} with respect to $A'$, then $\varphi^+(A\setminus A')=0$ by Proposition~\ref{prop_loc1Mprime}, and thus $\Phi(\mu, A)=\Phi(\mu, A')$,
which implies the thesis.
\end{proof}

We are now finally able to formulate the
desired $\Gamma$-convergence result.

\begin{theorem}\label{th_loc1Gamloc}
Assume $n\geq 2$, $\varphi^+ \ll {\mathcal L}^n$ and
for every minimizer $M$ of~\eqref{eq_loc1lim0P} one has
\[
M\cap \supp\varphi^-=\emptyset
\]
(in particular, this is true when any of the conditions of Proposition~\ref{prop_loc1chargeLim1} holds). Then the sequence of
functionals $\{G_k\}_{k=1}^\infty$ when $k\to \infty$
$\Gamma$-converges with respect to
the $*$-weak convergence of measures, to the functional
$G_\infty$.
\end{theorem}

The proof of the above Theorem~\ref{th_loc1Gamloc} will be quite lengthy and hence will be separated in a series of lemmata given in the section below.
We will now concentrate on the consequences of this theorem.

\begin{corollary}\label{co_loc1Gamloc1}
Assume $n\geq 2$, $\varphi^+ \ll {\mathcal L}^n$, call $f^+$ its density, and suppose that
\[
M\cap \supp\varphi^-=\emptyset
\]
for every minimizer $M$ of~\eqref{eq_loc1lim0P}
(this is true, in particular, under any of the conditions of Proposition~\ref{prop_loc1chargeLim1}).
Then every sequence of minimizers $\Sigma_k$ has a subsequence
(still called $\Sigma_k$) such that, for $k\to \infty$ one has
\begin{itemize}
\item[(i)]
$\Sigma_k\to M$ in Hausdorff distance for some minimizer $M$ to~\eqref{eq_loc1lim0P},
\item[(ii)]
$\mu_{\Sigma_k}\rightharpoonup \rho\, dx$ in $*$-weak sense of measures, where
\begin{equation}\label{eq_loc1_mulim}
\rho :=\lambda\int_M \left(f^+\right)^{\frac{n}{n+1}}(x)\, dx,
\end{equation}
where $\lambda$ is the normalizing coefficient
\[
\lambda := \left(\int_M \left(f^+\right)^{\frac{n}{n+1}}(x)\, dx\right)^{-1}.
\]
\end{itemize}
\end{corollary}

\begin{proof}
For a given minimizer $A$ to
to~\eqref{eq_loc1lim0P}
define
\[
H(A):=\inf\left\{\Phi(\mu,A)\ :\ \mu\in{\mathcal P}(D),\ \supp\mu\subset A
\right\}.
\]
One clearly has
\[
H(A)=\Phi(\mu', A),
\]
where $\mu'=\rho'\, dx$ with
\[
\rho' :=\lambda' \left(f^+\right)^{\frac{n}{n+1}}(x),
\]
and $\lambda'$ is the normalizing coefficient
\[
\lambda' := \left(\int_A \left(f^+\right)^{\frac{n}{n+1}}(x)\, dx\right)^{-1}.
\]
Hence,
\[
H(A)= \theta_n \left(\int_A (f^+(x))^{\frac{n}{n+1}}\, dx\right)^{\frac{n-1}{n}}.
\]
Let now $\{A_k\}$ be a minimizing sequence for $H$
of canonical minimizers to~\eqref{eq_loc1lim0P}.
Since $A_k^c$ are contained in a big ball, one has
that $A_k\to A$ in Hausdorff distance as $k\to \infty$. Clearly,
$A$ is still a minimizer of~\eqref{eq_loc1lim0P}, and by Proposition~\ref{prop_loc1Gamloc_A3_0b} one has
$1_{A_k}\to 1_A$ pointwise.
Hence, $H(A_k)\to H(A)$ as $k\to \infty$, which means that $H$ admits a minimizer.

Observe that
\begin{equation}\label{eq_loc1_infAmu}
\begin{aligned}
\inf\{ G_\infty & (\mu)\,: \, \mu \in {\mathcal P}(D)\}=\\
& \inf\left\{ \Phi(\mu, A)\,:\, \mu \in {\mathcal P}(D), \supp\mu\subset A, \, A \mbox{ is a minimizer to~\eqref{eq_loc1lim0P}}
\right\}=\\
& \inf\left\{ H(A)\,:\,A\mbox{ is a minimizer to~\eqref{eq_loc1lim0P}}
\right\}.
\end{aligned}
\end{equation}
Consider now an arbitrary sequence of minimizers $\Sigma_k$. By general properties of $\Gamma$-convergence it has a subsequence
(not relabeled) such that, for $k\to \infty$ one has
\[
\mu_{\Sigma_k}\rightharpoonup \mu,
\]
where $\mu$ is a minimizer of $G_\infty$. By~\eqref{eq_loc1_infAmu} the latter is supported on some minimizer $M$ to $H$ (which is hence, in particular, the minimizer of~\eqref{eq_loc1lim0P}), and minimizes
$\Phi(\cdot, M)$, so that~\eqref{eq_loc1_mulim} is valid.
\end{proof}

\section{Proof of Theorem~\ref{th_loc1Gamloc}}\label{sec_loc1limdens1}

\subsection{$\Gamma-\liminf$ inequality}\label{subsec_loc1liminf}

First, we will deal with the inequality for $\Gamma-\liminf$ given by the following statement.

\begin{proposition}\label{prop_loc1Gamloc_liminf}
Assume $n\geq 2$, $\varphi \ll {\mathcal L}^n$,
and for every minimizer $M$ of~\eqref{eq_loc1lim0P} one has
\[
M\cap \supp\varphi^-=\emptyset,
\]
(in particular, this is true when any of the conditions of Proposition~\ref{prop_loc1chargeLim1} holds). Suppose that, for a certain sequence $\Sigma_k$, one has $\mu_{\Sigma_k}\rightharpoonup \mu.$

Then
\begin{align*}
\liminf_k k^{1/n} & \left(
F(\Sigma_k)- \inf_A F(A)
\right)\geq\\
& \inf\left\{ \Phi(\mu, A)\,:\, \supp\mu\subset A, \, A \mbox{ is a minimizer to~\eqref{eq_loc1lim0P}}
\right\},
\end{align*}
\end{proposition}

To prove the above Proposition~\ref{prop_loc1Gamloc_liminf}, we make some auxiliary constructions. First of all notice that it is only necessary to prove the statement when any $\Sigma_k$ converges, up to subsequences, to a minimizer of~\eqref{eq_loc1lim0P} (if it is not the case, the term $F(\Sigma_k)- \inf_A F(A)$ does not tend to $0$ and hence the left hand side in the inequality tends to $+\infty$). Let us suppose, hence, $\Sigma_k \to M'$ in the sense of Hausdorff,
where $M'$ is a minimizer of~\eqref{eq_loc1lim0P} $\mu_k:=\mu_{\Sigma_k}\rightharpoonup \mu$ as $k\to \infty$, hence
$\mu$ is concentrated on $M'$. Let $M$ be the canonical minimizer of~\eqref{eq_loc1lim0P} with respect to $M'$.
Further, notice that by Lemma~\ref{lm_F_lim} one may assume without loss of generality that
all $\Sigma_k$ are contained in some
ball.

We now approximate the measure $\varphi^-$ by atomic measures $\varphi^- _j$
with $\varphi^- _j(\R^n)\leq \varphi^-(\R^n)$
in $*$-weak sense, i.e.\ so that
$\varphi^- _j\rightharpoonup \varphi^-$ as $j\to \infty$, in the following way. We cover $\R^n$ by a uniform grid $G _j$ of step $\varepsilon _j>0$, let the finite set $\{x_i\}_{i\in I ^j}$ be made of all such points in the cells $C_i$ of this grid that
$x_i\in \supp\varphi\cap C_i$
(hence $c_i :=\varphi^-(C_i) >0$) for all $i\in I ^j$, and let
\[
\varphi^- _j:=\sum_{i\in I ^j} c_i\delta_{x_i}.
\]
Keeping in mind  that
\[
M=\left(\bigcup_{x\in\supp \varphi^-} B_{r_{x}} (x)\right)^c,
\]
where $r_x:= \dist(x, M)$,
set
\[
M _j:=\left(\bigcup_{i\in I ^j} B^c_{r_{x_i}} (x_i)\right)^c.
\]
We have now the following easy statement.

\begin{lemma}\label{lm_loc1Gamloc_FkF}
One has $F _j(M _j)\leq \min_A F _j(A) + e _j$ and $F _j(M _j)\leq \min_A F (A) + e _j=F(M)+e_j$, where
\begin{align*}
    F _j(A) &:= \int_{\R^n} \dist(x, A)\, d(\varphi^+(x)-\varphi^- _j(x)),
\end{align*}
and
$e _j\to 0$ as $j\to \infty$.
\end{lemma}

\begin{proof}
We have
\begin{align*}
F _j(M _j) &= F _j(M) + \int_{\R^n} (\dist(x, M _j)-\dist(x, M))\, d(\varphi^+(x)-\varphi^- _j(x)) \\
& \leq F _j(M) + C d_H(M _j,M),
\end{align*}
where $C>0$ is independent of $j$. But since
\begin{equation}\label{eq_loc1FkFwass}
|F _j(A)-F(A)|\leq W_1(\varphi^- _j,\varphi),
\end{equation}
where $W_1$ stands for the Kantorovich-Wasserstein distance between measures,
we get
\begin{align*}
F _j(M _j) & \leq F(M) + W_1(\varphi^- _j,\varphi) + C d_H(M _j,M)\\
& =
\min_A F(A) + W_1(\varphi^- _j,\varphi) + C d_H(M _j,M).
\end{align*}
But, again, by~\eqref{eq_loc1FkFwass},
\[
|\min_A F _j (A)- \min_A F(A)|\leq W_1(\varphi^- _j,\varphi),
\]
and thus
\[
F _j(M _j)\leq \min_A F _j (A) + 2 W_1(\varphi^- _j,\varphi) + C d_H(M _j,M),
\]
which concludes the proof by setting
\[
e _j:=2 W_1(\varphi^- _j,\varphi) + C d_H(M _j,M)
\]
and keeping in mind  that $W_1(\varphi^- _j,\varphi^-)\to 0$ (since $\varphi^- _j\rightharpoonup \varphi^-$) and $d_H(M _j,M)\to 0$ (by Lemma~\ref{lm_loc1Gamloc_MkM}) as $j\to \infty$.
\end{proof}

From now on we fix a sequence $j_k\to\infty$ of indices and consider only the sets $M_{j_k}$. This sequence will be chosen so as to guarantee that the convergence $e_{j_k}\to 0$ is quick enough, according to some criteria to be made precise later.

Let now
\begin{align*}
r_i & := r_{x_i},\\
\varepsilon_i^k & := \dist(x_i, \Sigma_k)-\dist(x_i, M_{j_k})
\end{align*}
for every $i\in I^j$. The following statement holds true (independently of the convergence speed).

\begin{lemma}\label{lm_loc1Gamloc_A1add}
Letting
\begin{align*}
\varepsilon_k^\pm&:=\max\left\{
(\dist(y,\Sigma_k)-\dist(y,M_{j_k}))^\pm\,:\,y\in\supp\varphi^-_{j_k}\right\},
\end{align*}
where $f(\cdot)^\pm$ stand for the positive and negative part of the function
$f(\cdot)$ respectively,
we have $\varepsilon_k^\pm\to0$ as $k\to\infty$.
\end{lemma}

\begin{proof}
Suppose first that, up to a subsequence (not relabeled), $\varepsilon_k^+\to 2d>0$ as $k\to\infty$. This means the existence of $x_k\in\supp\varphi^-_{j_k}$ such that
\[
\dist(x_k, \Sigma_k)\geq \dist(x_k,  M_{j_k})+d
\]
for all sufficiently large $k\in\N$. Again up to a subsequence (not relabeled) one has $x_k\to x\in\supp\varphi^-$, and hence, keeping in mind  the convergence of $\Sigma_k$ to $M'$ and of $ M_{j_k}$ to $M$ in the Hausdorff distance and passing to a limit in the above inequality as $k\to\infty$, we get
\[
\dist(x,M')\ge\dist(x,M)+d,
\]
which is impossible since $\dist(y,M')=\dist(y,M)$ for all $y\in\supp\varphi^-$ by Proposition~\ref{prop_loc1Mprime}. This contradiction proves $\varepsilon_k^+\to0$ as $k\to\infty$. The proof of $\varepsilon_k^-\to0$ is completely symmetric.
\end{proof}

\begin{remark}
In our construction one has $\supp\varphi^-_{j}\subset\supp\varphi^-$. However, it is worth noting that the proof of the above Lemma uses only a milder property, namely, if $x_k\in \supp\varphi^-_{j}$ and $x_k\to x$ as $k\to \infty$, then $x\in\supp\varphi^-$, and hence the statement is still true in this case.
\end{remark}

Let us also let
\begin{align*}
\hat{\Sigma}_k & := \pi_{M_{j_k}}(\Sigma_k) \cup \partial M_{j_k}
\cup \left(\bigcup_i \left( \bar B_{r_i+\varepsilon_i^k}(x_i) \setminus B_{r_i}(x_i) \right)
\cap M_{j_k}\right),\\
\hat{M}_k &:=\bigcap_i B_{r_i+\varepsilon_i^k}(x_i)^c,
\end{align*}
see Figure \ref{loc1_fig1}.

\begin{center}
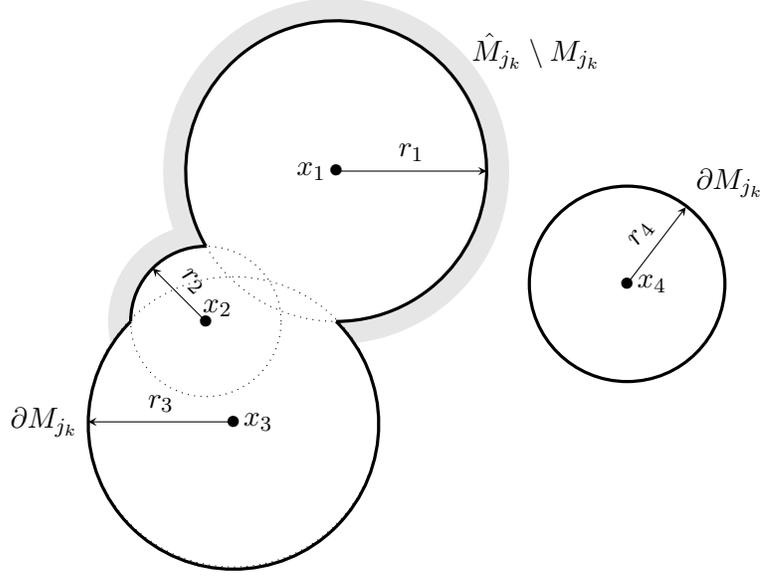
\begin{figure}
\begin{tikzpicture}
\fill[color=gray!20]  (0,-1) circle (1.3) ;
\fill[color=gray!20]  (1.732,1) circle (2.3) ;
\fill[white]  (0,-1) circle (1) ;
\fill[white]  (1.732,1) circle (2) ;
\fill[white] (0.366,-2.336) circle (1.931) ;

\draw (3.4,2.6) node[right] {$\hat M_{j_k}\setminus M_{j_k}$};

\draw[very thick] (0,0) arc (90:180:1) ;
\draw [dotted] (0,-1) circle (1) ;
\draw [->,>=stealth] (0,-1) -- (-0.707, -0.293)  node[midway,above,sloped] {$r_2$} node[near start,right] {$x_2$};
\draw (0,-1)  node{$\bullet$};

\draw[very thick]  (0,0) arc (210:-90:2) ;
\draw  [dotted]  (1.732,1) circle (2) ;
\draw [->,>=stealth]  (1.732,1)  --  (3.732,1) node[midway,above,sloped] {$r_1$};
\draw (1.732,1) node[left] {$x_1$} node{$\bullet$};

\draw[very thick]   (1.732,-1) arc (45:-225:1.931) ;
\draw  [dotted] (0.366,-2.336) circle (1.931) ;
\draw [->,>=stealth]  (0.366,-2.336)  --  (-1.565, -2.336) node[midway,above,sloped] {$r_3$} ;
\draw  (0.366,-2.336) node[right] {$x_3$} node{$\bullet$};
\draw (-1.565, -2.336) node[left] {$\partial M_{j_k}$};

\draw[very thick] (5.6,-0.5) circle (1.3) ;
\draw [->,>=stealth]  (5.6,-0.5)  --  (6.38, 0.52)  node[midway,above,sloped] {$r_4$};
\draw (5.6,-0.5) node[right] {$x_4$} node{$\bullet$};
\draw (6.38, 0.52) node[above right] {$\partial M_{j_k}$};

\end{tikzpicture}
\caption{The construction of $\hat M_k$}
\label{loc1_fig1}
\end{figure}
\end{center}

We need the following auxiliary statement.

\begin{lemma}\label{lm_loc1Gamloc_A2}
One has
\[
F_{j_k}(\Sigma_k)-  F_{j_k}(\hat{\Sigma}_k) \geq
F_{j_k}(\hat{M}_k)-  F_{j_k}(M_{j_k}).
\]
In particular, keeping in mind  Lemma~\ref{lm_loc1Gamloc_FkF}, one has
\[
F_{j_k}(\Sigma_k)-F_{j_k}(\hat{\Sigma}_k)\ge-e_{j_k}.
\]
\end{lemma}

\begin{proof}
If $x_i\in \supp\varphi^-_k$, then
$$\left\{
\begin{array}{ll}
(r_i+\varepsilon_i^k)-r_i=
\dist(x_i, \Sigma_k)-\dist(x_i, \hat{\Sigma}_k)\\
(r_i+\varepsilon_i^k)-r_i=
\dist(x_i, \hat{M}_k)-\dist(x_i, M_{j_k}),
\end{array}
\right.$$
hence
\[
\dist(x_i, \Sigma_k)-\dist(x_i, \hat{\Sigma}_k)=
\dist(x_i, \hat{M}_k)-\dist(x_i, M_{j_k}).
\]
It suffices hence to verify
\begin{equation}\label{eq_loc1A2a}
\dist(x, \Sigma_k)-\dist(x, \hat{\Sigma}_k)\geq
\dist(x, \hat{M}_k)-\dist(x, M_{j_k})
\end{equation}
for all $x\in \varphi^+$.
For this purpose consider the following possible cases.

{\sc Case A.} If $x\not\in M_{j_k}$, then $\dist(x, \hat{\Sigma}_k)=\dist(x, M_{j_k})$,
because $\partial M_{j_k}\subset \hat{\Sigma}_k\subset M_{j_k}$.
But $\Sigma_k\subset \hat{M}_k$; hence $\dist(x, \Sigma_k)\geq \dist(x, \hat{M}_k)$, which shows~\eqref{eq_loc1A2a} for this case.

{\sc Case B.} If $x\in M_{j_k}\cap \hat{M}_k$, then $\dist(x, \hat{M}_k)=\dist(x, M_{j_k})=0$. Let $y\in \Sigma_k$ be such that
$\dist(x, \Sigma_k)=|x-y|$.
If $y\in M_{j_k}$, then by the definition of $\hat{\Sigma}_k$ one has $y\in \hat{\Sigma}_k$, which implies
$\dist(x, \hat{\Sigma}_k)\leq |x-y|=\dist(x, \Sigma_k)$,
which shows~\eqref{eq_loc1A2a}. Otherwise, if
$y\not\in M_{j_k}$, then
\[
\dist(x,\partial M_{j_k})\leq |x-y|=\dist(x,\Sigma_k),
\]
and hence, since $\partial M_{j_k}\subset\hat{\Sigma}_k$,
\[
\dist(x,\hat{\Sigma}_k)\le\dist(x,\partial M_{j_k})\le\dist(x,\Sigma_k),
\]
which again shows~\eqref{eq_loc1A2a}.

{\sc Case C.} Finally, if $x\in M_{j_k}\setminus\hat{M}_k$, then $\dist(x, \hat{\Sigma}_k)=\dist(x, M_{j_k})=0$, and to prove~\eqref{eq_loc1A2a}
it suffices to verify $\dist(x,\Sigma_k)\geq \dist(x,\hat{M}_k)$.
The latter relationship is however valid because
$\Sigma_k\subset \hat{M}_k$ by construction.
This completes the proof.
\end{proof}

We are now able to prove Proposition~\ref{prop_loc1Gamloc_liminf}.

\begin{proof}[Proof of Proposition~\ref{prop_loc1Gamloc_liminf}]
We construct new sets $\tilde{\Sigma}_k$ with $\#\tilde{\Sigma}_k <\infty$
which ``approximate well'' the sets $\hat{\Sigma}_k$ (since the latter are not finite sets). This will be done as follows. Let $\varepsilon_k:=\sup_{i\in I^k}(\varepsilon_i^k)^+$ (recall that $\varepsilon_k\to0$ as $k\to\infty$ by Lemma~\ref{lm_loc1Gamloc_A1add}), and let $\delta_k\ge0$ and $\alpha>0$ to be chosen later. According to Corollary~\ref{lm_loc1Gamloc_A3} there exists a $k^{-\alpha}$-net $\mathcal{G}_{Surf}(M_{j_k},k^{-\alpha})$ inside $\partial M_{j_k}$ such that
\[
\# \mathcal{G}_{Surf}(M_{j_k}, k^{-\alpha})\leq Ck^{-\alpha(n-1)}.
\]
Further, by Lemma~\ref{lm_loc1Gamloc_A4} we may construct inside the set
\[
D:=\bigcup \left( \bar B_{r_i+\varepsilon_i^k}(x_i)\right) \setminus \bigcup\left( B_{r_i}(x_i) \right)
\]
a $\delta_k$-net $\mathcal{G}_{2}\subset\mathcal{G}_{Vol}(M_{j_k},\varepsilon_k, \delta_k)$ satisfying
\[
\# \mathcal{G}_{2}\leq \frac{C(\varepsilon_k+\delta_k)}{\delta_k^n}.
\]
Here and below $C>0$ will stand for a constant independent on $k$ possibly differing from line to line.

After defining these two grids we need to handle the projection part in the definition of $\hat{\Sigma}_k$. In order to preserve the same number of points of $\Sigma_k$ and to reproduce the measure $(\pi_{M_{j_k}})_\#\mu_{\Sigma_k}$, this projection will be replaced in $\tilde{\Sigma}_k$ with a set $\Sigma_k'$, obtained by adding to $\pi_{M_{j_k}}(\Sigma_k)$ a finite set of  points in the following way. Take $\Sigma_k'$ as the union of $\pi_{M_k}(\Sigma_k)$
with some finite sets $D_k(y)$ of cardinality
$\#\pi_{M_k}^{-1}(y)-1$ arbitrarily chosen very close to $y$ for each
$y\in \pi_{M_k}(\Sigma_k)$ such that $\#\pi_{M_k}^{-1}(y)>1$ (in particular we take them in $B_{1/k^2}(y)\cap \partial M_k$)
in such a way that $D_k(y)\cap \pi_{M_k}(\Sigma_k)=\emptyset$. It is clear that in this way we can guarantee
 \begin{equation}\label{eq_loc1_sk1}
W_1(\mu_{\Sigma_k'}, (\pi_{M_{j_k}})_\#\mu_{\Sigma_k}) \leq 1/k. 
\end{equation}
In order to justify~\eqref{eq_loc1_sk1}, just take an arbitrary $u\in \mbox{Lip}_1(U)$, where $U$ is the tubular neighborhood
of $\partial M_k$, we have
\begin{align*}
 \left|\int_U u\, d \mu_{\Sigma_k'}- \int_U u\,\pi_{M_k\#}(\mu_{\Sigma_k})
 \right| &=
 \left|\sum_{\overset{y\in \pi_{M_k}(\Sigma_k),}{ \#\pi_{M_k}^{-1}(y)>1}}\;
 \sum_{x\in D_k(y)} (u(x)-u(y))
 \right|\\
 &\leq \sum_{\overset{y\in \pi_{M_k}(\Sigma_k),}{ \#\pi_{M_k}^{-1}(y)>1}}\;
 \sum_{x\in D_k(y)} \left| \,u(x)-u(y)\right|\\
  &\leq \sum_{\overset{y\in \pi_{M_k}(\Sigma_k),}{ \#\pi_{M_k}^{-1}(y)>1}}\; Ê \sum_{x\in D_k(y)} \frac{1}{k^2}\leq k \frac{1}{k^2}= \frac 1 k.
\end{align*}


We substitute now, in the construction of $\hat{\Sigma}_k$, the set $D$ with
$\mathcal{G}_2$ and $\partial M_{j_k}$ with $\mathcal{G}_{Surf}( M_{j_k}, k^{-\alpha})$, namely,
\begin{align*}
\tilde{\Sigma}_k & := \Sigma_k' \cup \mathcal{G}_{Surf}(M_{j_k}, k^{-\alpha})
\cup (\mathcal{G}_2\cap M_{j_k}).
\end{align*}
Hence, using the estimates on $\# \mathcal{G}_{Surf}(M_{j_k}, k^{-\alpha})$ and on
$\# \mathcal{G}_2$,
we get
\[
\#\tilde{\Sigma}_k \leq \#\Sigma_k + C k^{\alpha(n-1)} +
C\frac{\varepsilon_k+\delta_k}{\delta_k^n}.
\]
Choose now $\delta_k$ such that $k\delta_k^n=\sqrt{\varepsilon_k+\delta_k}$. This equation admits, for fixed $k$ and $\varepsilon_k$, a unique solution $\delta_k\geq 0$. One also easily checks that it implies
\[
\delta_k\leq \varepsilon_k\vee C k^{1/(1/2-n)}
\]
(since either $\delta_k\leq \varepsilon_k$, or $k\delta_k^n\leq \sqrt{2\delta_k}$),
and in particular $\lim_{k\to \infty} \delta_k=0$. In this way one has
\begin{equation}\label{cond on k e deltak}
\frac {1}{k} \frac{\varepsilon_k+\delta_k}{\delta^n_k}=\sqrt{\varepsilon_k+\delta_k}\to 0,\;\mbox{ and }\; k^{1/n}\delta_k=(\varepsilon_k+\delta_k)^{1/2n}\to 0.
\end{equation}
Notice hence that for every $0<\alpha<1/(n-1)$ one has
\begin{equation}\label{eq_loc1_ok1}
\#\mathcal{G}_{Surf}(M_{j_k},k^{-\alpha})\leq Ck^{\alpha(n-1)} = o(k),\mbox{ and }
\#\mathcal{G}_2\leq\frac{C(\varepsilon_k+\delta_k)}{\delta_k^n}=o(k),
\end{equation}
which gives
\begin{equation}\label{eq_loc1_ok2}
\#\tilde{\Sigma}_k = \#\Sigma_k + o(k),
\end{equation}
as $k\to \infty$. By Lemma~\ref{lm_loc1Gamloc_A5} below combined
with~\eqref{eq_loc1_sk1}, one has then
\begin{equation}\label{eq_loc1_ok3}
\mu_{\tilde{\Sigma}_k}\rightharpoonup \pi_{M\#}\mu=\mu
\end{equation}
(the latter equality being justified by the fact that $\mu$ is concentrated on $M$).
On the other hand,
\[
 F_{j_k}(\Sigma_k)-  F_{j_k}(\hat{\Sigma}_k) \geq
- e_k
\]
by Lemma~\ref{lm_loc1Gamloc_A2}, and
\[
F_{j_k}(\hat{\Sigma}_k)-F_{j_k}(\tilde{\Sigma}_k)\ge
-Ck^{-\alpha}-C\delta_k
\]
by construction. Keeping in mind  that
\begin{gather*}
| F_{j_k}(\Sigma_k)-F(\Sigma_k)|\leq e_{j_k},\\
F(M)=\min_A F(A)\leq \min_A F_{j_k}(A)+e_{j_k}\leq F_{j_k}(M_{j_k})+e_{j_k}
\end{gather*}
(see Lemma~\ref{lm_loc1Gamloc_FkF}), we get from the above inequalities
\begin{equation}\label{eq_loc1_ok1n}
k^{1/n}\big(F(\Sigma_k)-F(M)\big)
\ge k^{1/n}\big(F_{j_k}(\tilde{\Sigma}_k)-F_{j_k}(M_{j_k})\big)
+ k^{1/n}(e_k + c\delta_k + C k^{-\alpha}).
\end{equation}

Choose now $1/n<\alpha<1/(n-1)$ so that~\eqref{eq_loc1_ok1} and hence~\eqref{eq_loc1_ok2} and~\eqref{eq_loc1_ok3} still hold.
By choosing a fast enough subsequence $j_k$  we may assume
$k^{1/n} e_{j_k}\to 0$ as $k\to \infty$.
Recall also that $k^{1/n} \delta_k\to 0$ as a consequence of \eqref{cond on k e deltak}, and hence, combining~\eqref{eq_loc1_ok1n} with~\eqref{eq_loc1_ok2} we get
\begin{equation}\label{eq_loc1_G1}
\liminf_k k^{1/n}\left(
 F(\Sigma_k)-  F(M)
\right)\geq
\liminf_k (\#\tilde{\Sigma}_k)^{1/n} \left(
 F_{j_k}(\tilde{\Sigma}_k)-  F_{j_k}(M_{j_k})
\right).
\end{equation}
But
\begin{align*}
 F_{j_k}(\tilde{\Sigma}_k)-   F_{j_k}(M_{j_k})
= & \int_{M_{j_k}}\dist(x, \tilde{\Sigma}_k)\, d\varphi^+(x) + \\
& \int_{M_{j_k}^c}\left(\dist(x, \tilde{\Sigma}_k)-\dist(x, M_{j_k})\right)\, d(\varphi^+(x)-\varphi^-_{j_k}(x),
\end{align*}
and since
\[
|\dist(x, \tilde{\Sigma}_k)-\dist(x, M_{j_k})|\leq k^{-\alpha}
\]
for every $x\in M_{j_k}^c$
(because $\mathcal{G}_{Surf}(M_{j_k}, k^{-\alpha})\subset \tilde{\Sigma}_k$),
then
we have
\begin{equation}\label{eq_loc1_G2}
k^{1/n} \left( F_{j_k}(\tilde{\Sigma}_k)-  F_{j_k}(M_{j_k})\right)
\geq \int_{M_{j_k}}\dist(x, \tilde{\Sigma}_k)\, d\varphi^+(x) -
C k^{1/n - \alpha}.
\end{equation}
Plugging~\eqref{eq_loc1_G2} into~\eqref{eq_loc1_G1} and keeping in mind  that
$k^{1/n - \alpha}=o(1)$ as $k\to \infty$ by our choice of $\alpha$, we arrive at
the inequality
\[
\liminf_k k^{1/n} \left(
 F_{j_k}(\Sigma_k)-  F_{j_k}(M)
\right)\geq
\liminf_k (\# \tilde{\Sigma}_k)^{1/n} \int_{M_{j_k}}\dist(x, \tilde{\Sigma}_k)\, d\varphi^+(x).
\]

Since $M\subset M_{j_k}$ by construction, we have
$\varphi^+\res M_{j_k}\geq \varphi^+\res M$, and hence
\[
\liminf_k k^{1/n} \left(
 G_k(\Sigma_k)-  G_k(M)
\right)\geq
\liminf_k (\# \tilde{\Sigma}_k)^{1/n} \int_M\dist(x, \tilde{\Sigma}_k)\, d\varphi^+(x).
\]
By Lemma~\ref{lm_TilliMoscGammaloc} applied with $\tilde\Sigma_k$ in place
of $\Sigma_k$, one has
\[
\liminf_k (\# \tilde{\Sigma}_k)^{1/n} \int_M\dist(x, \tilde{\Sigma}_k)\, d\varphi^+(x) \geq \Phi(\mu,M).
\]
Hence,
\begin{align*}
\liminf_k k^{1/n} & \left(
F(\Sigma_k)- \inf_A F(A)
\right)\geq\\
& \inf\left\{ \Phi(\mu, A)\,:\, \supp\mu\subset A, \, A \mbox{ is a canonical minimizer to~\eqref{eq_loc1lim0P}}
\right\}.
\end{align*}
It remains to invoke Lemma~\ref{lm_loc1_canonFinf1}, which gives the possibility to write
\begin{align*}
\liminf_k k^{1/n} & \left(
F(\Sigma_k)- \inf_A F(A)
\right)\geq
G_\infty (\mu),
\end{align*}
thus concluding the proof.
\end{proof}

\begin{lemma}\label{lm_loc1Gamloc_A5}
If $\mu_k:=\mu_{\Sigma_k}\rightharpoonup\mu$, then $\pi_{M_{j_k}\#}\mu_k\rightharpoonup\pi_{M\#}\mu$.
\end{lemma}

\begin{proof}
There is an $R>0$ such that all $M_{j_k}$ and $M$ satisfy the $R$-uniform external ball condition. Let $U$ be an $R/2$-tubular neighborhood of $M$, i.e. $U:=(M)_{R/2}\setminus M$. In view of Lemma~\ref{lm_loc1Gamloc_A1add} one has $\varepsilon_k^-\to 0$ as $k\to \infty$, and therefore $\supp\mu_k\subset U$ for all sufficiently large $k\in\N$.
Further, the projection maps $\pi_{M_{j_k}}$ 
defined over $U$ converge uniformly over compact sets to $\pi_M$. To prove that it is enough to consider a sequence $x_k\to x$ and remark that $\pi_{M_{j_k}}(x_k)$ has a limit up to subsequences (since it is bounded). Call $y$ such a limit: it satisfies
$$d(x,M)=\lim_{k\to\infty} d(x_k,M_{j_k}))=\lim_{k\to\infty} |x_k-\pi_{M_{j_k}}(x_k)|=|x-y|,$$
which proves $y=\pi_M(x)$ (because the projection is unique for $x\in \bar U$). Hence, the limit being unique, we proved $\pi_{M_{j_k}}(x_k)\to \pi_M(x)$, which is equivalent to the uniform convergence over compact sets.

Thus, recalling that all
$\mu_k$ are assumed to be
concentrated over some ball, for every bounded $f\in C(\bar U)$ one has
\begin{align*}
\int_U f(x)\, d \pi_{M_{j_k}\#}\mu_k (x)&= \int_U f(\pi_{M_{j_k}}(x))\, d \mu_k(x)\to\\
& \int_U f(\pi_M(x))\, d \mu(x) = \int_U f(x)\, d \pi_{M\#}\mu (x),
\end{align*}
as $k\to \infty$, which is the desired assertion.
\end{proof}

\subsection{$\Gamma-\limsup$ inequality}\label{subsec_loc1limsup}

Now we prove the inequality for $\Gamma-\limsup$.

\begin{proposition}\label{prop_loc1Gamloc_limsup}
Assume $n\ge2$ and $\varphi\ll{\mathcal L}^n$.
Then, for every fixed $\mu$, there exists a sequence $\Sigma_k$ such that $\mu_{\Sigma_k}\deb\mu$ and
\begin{align*}
\limsup_k k^{1/n} & \left(
F(\Sigma_k)- \inf_A F(A)
\right)\leq G_\infty(\mu).
\end{align*}
\end{proposition}

\begin{proof}
Recall that due to Lemma~\ref{lm_loc1_canonFinf1}
one has
\begin{align*}
G_\infty(\mu) & = \inf\left\{ \Phi(\mu, A)\,:\, \supp\mu\subset A, \, A \mbox{ is a canonical minimizer to~\eqref{eq_loc1lim0P}}
\right\},
\end{align*}
and choose a canonical minimizer $M$ to~\eqref{eq_loc1lim0P} such that
\[
\Phi(\mu, M)\leq G_\infty(\mu)+\varepsilon.
\]
We make use of the constructions of the sets $M_{j}$ made in Subsection~\ref{subsec_loc1liminf}. Choose (up to passing to a subsequence of $k$) the sets $M_{j_k}$ such that
\[
d_H(\partial M_{j_k}, \partial M)\leq k^{-\alpha}.
\]
Let $\tilde\Sigma_k\subset M$ be such that
\[
\limsup_k (\#(\tilde\Sigma_k)^{1/n} \int_M \dist(x, \tilde\Sigma_k)\, d\varphi^+(x)\leq
\Phi(\mu, M),
\]
where $\#(\tilde\Sigma_k)\to\infty $ will be chosen in a moment.
Let also
\[
\Sigma_k:=\tilde\Sigma_k \cup \mathcal{G}_{Surf}(M_{j_k}, k^{-\alpha}),
\]
where $ \mathcal{G}_{Surf}(M_{j_k}, k^{-\alpha})$ is a $k^{-\alpha}$-net inside
$\partial M_{j_k}$ constructed again according to Lemma~\ref{lm_loc1Gamloc_A3}.

Take now  $\tilde\Sigma_k$ with $\#\tilde\Sigma_k$ being such that
\[
\#\left(\tilde\Sigma_k\cup  \mathcal{G}_{Surf}(M_{j_k}, k^{-\alpha})\right)=k.
\]
Since $\#\mathcal{G}_{Surf}(M_{j_k}, k^{-\alpha}=o(k)$, this implies $\#(\tilde\Sigma_k)=k-o(k)$ and hence we also have
\[
\limsup_k k^{1/n} \int_M \dist(x, \tilde\Sigma_k)\, d\varphi^+(x)\leq
\Phi(\mu, M).
\]

By construction then
\[
|\dist(x, \Sigma_k)-\dist(x,M)|\leq 2 k^{-\alpha}
\]
for all $x\not \in M$. Hence,
$$
k^{1/n}\left|\int_{M^c}(\dist(x,\Sigma_k)-\dist(x,M))\,\varphi(x)\right|
\le Ck^{1/n-\alpha}=o(1)
$$
as $k\to\infty$. Thus
\begin{align*}
k^{1/n}\left(F(\Sigma_k)-\inf_A F(A)\right)
&=k^{1/n}\big(F(\Sigma_k)-F(M)\big)\\
&=k^{1/n}\int_M \dist(x, \tilde\Sigma_k)\,d\varphi^+(x)\\
&\quad+k^{1/n}\int_{M^c}(\dist(x,\Sigma_k)-\dist(x,M))\,\varphi(x)\\
&=k^{1/n}\int_M\dist(x,\tilde\Sigma_k)\,d\varphi^+(x)+o(1)
\end{align*}
as $k\to\infty$.
\end{proof}

%
%
%
%
%

\appendix
\section{Some properties of sets satisfying a uniform \\ external ball condition}\label{sec_loc1unifball}

In this section we collect some properties of sets satisfying a uniform external ball condition, specifically of those we are dealing in this paper.

\begin{definition}\label{def_loc1_uball}
We say that a closed set $M\subset\R^n$ satisfies the uniform $R$-external ball condition, for given $R>0$, if for every $x\in\partial M$ there is a ball $B_R(z)$ of radius $R$ touching $M$ at $x$, i.e. such that $\bar B_R(z)\cap M=\{x\}$.
If not necessary, the reference to $R$ will be omitted and we just speak about uniform external ball condition.
\end{definition}

We start with a rather weak result which however is proven here for the sake of completeness.

\begin{lemma}\label{lm_loc1_ub0}
Let $M$ be a set satisfying the $R$-uniform external ball condition. Then ${\mathcal L}^n(\partial M)=0$.
\end{lemma}

\begin{proof}
For every $x\in\partial M$, denoting by $B_R(0)$ the ball of radius $R$ touching $\partial M$ in $x$, we have
\[
\limsup_{r\to 0^+} \frac{{\mathcal L}^n (B_r(x)\cap M)}{\omega_n r^n}\leq
\limsup_{r\to 0^+} \frac{{\mathcal L}^n (B_r(x)\setminus B_R(0))}{\omega_n r^n} = \frac 1 2,
\]
which means that $x$ is not a Lebesgue point of the characteristic function $1_M$, and thus shows the claim.
\end{proof}

%
%
Further throughout this section let $K\subset \R^n$ be a compact set, and
let
\[
M:= \left(\bigcup_{x\in K} B_{r_x}(x)\right)^c,
\]
where $r_x:= \dist(x, M)$ satisfies the estimate
$r_x\geq R>0$ for some $R>0$ and for all $x\in K$.
Clearly $M$ satisfies the $R$-uniform external ball condition.

Let $\sigma_k\searrow 0$, and let $\{x_i\}_{i\in I^k}\subset K$ be
finite $\sigma_k$-nets of $K$.
\[
M_k:=\left(\bigcup_{i\in I^k} B^c_{r_{x_i}} (x_i)\right)^c.
\]
The following assertions hold.

\begin{lemma}\label{lm_loc1Gamloc_MkM}
One has $M_k \to M$ in the sense of Hausdorff as $k\to \infty$.
\end{lemma}

\begin{proof}
Let $y_k\in M_k$ and $y_k\to y$ as $k\to\infty$. Clearly, then $d(y_k,x_i)\geq r_{x_i}$ for all $i\in I^k$. But for all $x\in K$ there is an $x_{i_k}$ with $i_k\in I^k$ such that $d(x,x_{i_k})\leq \sigma_k$. Hence, $d(y_k,x_{i_k})\geq r_{x_{i_k}}$, and passing to a limit as $k\to\infty$ (mind that $x\mapsto r_x$ is continuous), we get
\[
|y-x|\geq r_x,
\]
for all $x\in K$,
which means that $y\in M$. To conclude the proof it remains to observe that
$M\subset M_k$.
\end{proof}

%

\begin{proposition}\label{prop_loc1Gamloc_A3_0a}
Let
\[
A:=\left(\bigcup_{x\in K} B_{r_x}(x)\right)^c,
\]
where $K\subset \R^n$ is a finite set (i.e.\ $\# K <\infty$)
and $r_x\geq R$ for some $R>0$, and for all $x\in K$.
Then
\[
\HH^{n-1} (\partial A)\leq C:=n\omega_n(\diam K+R')^n/R,
\]
where $R':=\max_{z\in K} r_z$ and $\omega_n$ stands for the volume
of the unit ball in $\R^n$.
\end{proposition}

\begin{proof}
For every set $S\subset \partial B_r(x)$ the volume ${\mathcal L}^n (T(S))$ of
the conical segment
\[
T(S):= \bigcup_{y\in S} [x,y]
\]
is given by
\[
{\mathcal L}^n (T(S)) = \HH^{n-1}(S)r/n.
\]
Letting
\[
S_z:=\partial B_{r_z}(z)\setminus \bigcup_{u\in K, u\neq z} B_{r_u}(u),
\]
we have that all the internal parts of $T(S_z)$ are disjoint and
\[
\bigcup_{z\in K} T(S_z)\subset A^c=
\bigcup_{z\in K} B_{r_z}(z)\subset B_{\diam K+R'}(x),
\]
where $x\in K$ is some (arbitrary) point of $K$.
Thus
\begin{align*}
\frac{R\HH^{n-1}(\partial A)}{n} & =\frac{R}{n} \sum_{z\in K} \HH^{n-1}(S_z)
\leq
\sum_{z\in K} \frac{ r_z \HH^{n-1}(S_z)}{n}\\
& ={\mathcal L}^n\left( \bigcup_{z\in K} T(S_z)\right)
\leq \omega_n(\diam K+R')^n,
\end{align*}
which gives the desired claim.
\end{proof}

\begin{lemma}\label{lm_loc1Gamloc_A7a}
One has $1_{M_k}(x)\to 1_M(x)$ for all $x\not\in\partial M$ (hence for a.e. $x\in\R^n$).
\end{lemma}

\begin{proof}
To show the first statement, denote
\[
1_A(x):=\limsup_k 1_{M_k}(x),\qquad
1_B(x):=\liminf_k 1_{M_k}(x).
\]
Consider now a point $x\not \in M$; from the Hausdorff convergence $M_k\to M$ we deduce that, for all sufficiently large values of $k$, $x\notin M_k$, hence
$1_B(x)=1_A(x)=1_M(x)=0$, which implies $B\subset A\subset M$. On the other hand, if $x\in M\setminus B$, then $x\not \in M_k$ for an infinite sequence
of $k$, hence $x$ cannot belong to the inner part of $M$, i.e. $x\in\partial M$. This shows $M\setminus B\subset\partial M$. Hence, for all $x\not \in \partial M$, one has
\[
\lim_k 1_{M_k}(x) =1_M(x)
\]
as claimed. The fact that this convergence is true a.e. over $\R^n$ follows from Lemma~\ref{lm_loc1_ub0}.
\end{proof}

\begin{corollary}\label{co_loc1Gamloc_A3_0a}
One has $1_{M_k^c}\to 1_{M^c}$ strongly in $L^1(\R^n)$ and weakly in $BV(\R^n)$ as $k\to\infty$. In particular, ${\mathcal L}^n(M_k\Delta M)\to0$. Setting $R':=\max_{z\in K}\dist(z,M)$, we have as a consequence
\[
\HH^{n-1} (\partial M)\leq C:=n\omega_n(\diam K+R')^n/R.
\]
\end{corollary}

\begin{proof}
Observe that
\[
M_k^c\subset \Omega:=B_{\diam K+R'}(y),
\]
where $R':=\max_{z\in K} \dist(z,M)$ and
$y\in K$ is an arbitrary point, so that
\[
\|1_{M_k^c}\|_1\leq C.
\]
Keeping in mind  that
\[
|D1_{M_k^c}|(\R^n)= \HH^{n-1} (\partial M_k)\leq C,
\]
by Proposition~\ref{prop_loc1Gamloc_A3_0a}, one has that the sequence $\{1_{M_k^c}\}$ is (weakly) compact in $BV(\Omega)$, hence strongly compact in $L^1(\Omega)$. Consider any convergent subsequence (not relabeled). It is converging in $L^1(\Omega)$, and the limit has to be a characteristic function of some set, which, as just proved, must be $1_{M^c}$. Hence the whole sequence $\{1_{M_k^c}\}$ is converging in $L^1(\Omega)$ to $1_{M^c}$, which is the first claim of the statement being proven.

As for the second claim, we use lower semicontinuity of the total variation, obtaining
\begin{align*}
\HH^{n-1} (\partial M)=|D1_{M^c}|(\R^n)&\leq \liminf_k |D1_{M_k^c}|(\R^n)\\ & =\liminf_k\HH^{n-1}(\partial M_k)\\
&\le\liminf_k n\omega_n(\diam K+R'_k)^n/R\\
&=n\omega_n(\diam K+R')^n/R,
\end{align*}
where $R'_k:=\max_{i\in I^k} r_{x_i}$.
\end{proof}

\begin{proposition}\label{prop_loc1Gamloc_A3_0b}
Assume that $\tilde M_k$ satisfy $\tilde M_k= \left(\bigcup_{x\in K} B_{\dist(x,\tilde M_k)}(x)\right)^c$, where for every $x\in K$ one has $\dist(x,\tilde M_k)\ge R$,
and $\tilde M_k\to \tilde M$ in Hausdorff distance as $k\to \infty$, for some $\tilde M\subset\R^n$. Let then
\[
\hat M:=\left(\bigcup_{x\in K}B_{r_x}(x)\right)^c,
\]
where $r_x:= \dist(x, \tilde M)$.
Then $r_x\geq R>0$ for all $x\in K$ and
\begin{itemize}
\item[(i)] $\tilde M\subset \hat M$, while $\hat M\setminus \tilde M\subset \partial\hat M$ (so in particular, $\mathcal{L}^n(\hat M\setminus \tilde M)=0)$;

\item[(ii)] $1_{\tilde M_k}(x)\to 1_{\tilde M}(x)$
for all $x\not\in \partial \tilde M$ (hence for a.e. $x\in \R^n$).
\end{itemize}
\end{proposition}

\begin{proof} The inclusion $\tilde M\subset \hat M$ is immediate by definition of $\hat M$.

Consider now an arbitrary $y\in \hat M$. One has $y\in B_{r_x}(x)^c$, hence
\[
d(y,x)\ge r_x\mbox{ for all }x\in K.
\]
Then
\begin{itemize}
\item[(a)] either $y\in \tilde M_k$ for a subsequence of $k$ (not relabeled), hence $y\in \tilde M$,
\item[(b)] or $y\not \in \tilde M_k$ (for all sufficiently large $k\in \N$), that is,
\[
d(y,x_k)< \dist (x_k, M_k)
\]
for some $x_k\in K$. Passing to a subsequence of $k$ (not relabeled), we have $x_k\to x\in K$, an hence, passing to a limit as $k\to\infty$ in the above estimate, we get
\[
d(y,x)\leq \dist (x, \tilde M)=\dist(x,\hat M)= r_x
\]
for some $x\in K$. Therefore, keeping in mind  that $y\in \hat M$, we get $d(y,x)=r_x$ for some $x\in K$ and $d(y,z)\geq r_z$ for all $z\in K$, which means $y\in \partial \hat M$.
\end{itemize}
This completes the proof of~(i).

To prove(ii), repeat word-for-word the proof of Lemma~\ref{lm_loc1Gamloc_A7a}
with $\tilde M_k$ instead of $M_k$ keeping in mind  that
now
\[
|D1_{\tilde M_k^c}|(\R^n)= \HH^{n-1} (\partial \tilde M_k)\leq C,
\]
by Proposition~\ref{prop_loc1Gamloc_A3_0a}, and acting
as in Corollary~\ref{co_loc1Gamloc_A3_0a} one shows
$\mathcal{H}^{n-1}(\tilde M)<+\infty$, and hence $\mathcal{L}^n(\tilde M)=0$.
\end{proof}

To clarify the above Proposition~\ref{prop_loc1Gamloc_A3_0b}, consider the following example showing that in general one cannot expect $\tilde M=\hat M$, but just $\tilde M\subset \hat M$.

\begin{example}\label{ex_loc1_Mhat}
Let $K$ be the boundary of a stadium with radius $1$, and $\tilde M_k$ be the complement of open stadia with radii larger than $2$, as drawn on Figure~\ref{stadium}: then the limit $\tilde M$ is the complement of the limit stadium of radius $2$, while $\hat M$ also includes the central segment.

\begin{center}
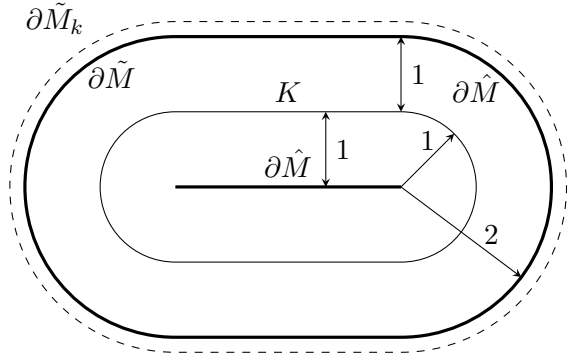
\begin{figure}[h]
\begin{tikzpicture}

\draw (0,0) arc(90:270:1) -- (3,-2) arc (-90:90:1) -- cycle;

\draw[very thick] (0,1) arc(90:270:2) -- (3,-3) arc (-90:90:2) -- cycle;

\draw[dashed] (0,1.2) arc(90:270:2.2) -- (3,-3.2) arc (-90:90:2.2) -- cycle;

\draw[very thick] (0,-1) -- (3,-1);

\draw [<->,>=stealth] (2,-1) -- (2,0) node[midway,right] {$1$};
\draw [<->,>=stealth] (3,0) -- (3,1) node[midway,right] {$1$};

\draw [->,>=stealth] (3,-1) -- (3.707,-0.293) node[midway,above] {$1$};
\draw [->,>=stealth] (3,-1) -- (4.6,-2.2) node[near end,above] {$2$};

\draw (1.5,0) node[above]{$K$};
\draw (1.5,-1) node[above]{$\partial \hat M$};
\draw (4,0) node[above]{$\partial \hat M$};
\draw (-1.6,0.9) node[above]{$\partial\tilde M_k$};
\draw (-1.3,0.5) node[right]{$\partial\tilde M$};

\end{tikzpicture}
\caption{An example with $\tilde M\neq \hat M$}
\label{stadium}
\end{figure}
\end{center}

\end{example}

\begin{lemma}\label{lm_loc1Gamloc_A4}
Let $A$ be as in Lemma~\ref{lm_loc1Gamloc_A3} and
\[
A_\varepsilon:=\left(A^c\right)_\varepsilon\cap A.
\]
Then, for every $\delta>0$ there is a $\delta$-net $\mathcal{G}_{Vol}(A,\varepsilon,\delta)\subset A_\varepsilon$
of $A_\varepsilon$,
such that
\[
\# \mathcal{G}_{Vol}(A,\varepsilon,\delta) \leq C\frac{\varepsilon+\delta}{\delta^{n}},
\]
for some $C>0$ depending only on $R$ and on $\diam \partial A$.
\end{lemma}

\begin{proof}
Consider a uniform cubic grid of step $\delta/\sqrt{n}$. It is sufficient to estimate the number of cubes in this grid intersecting $A_\varepsilon$, since then one can define $\mathcal{G}_{Vol}(A,\varepsilon,\delta)$ by picking one point on  $A_\varepsilon$ for every cube such that the intersection is non-empty.
These cubes are all contained in the region $\left(\bigcup_{x\in K} B_{r_x+\varepsilon+\delta}(x)\right)\cap\left(\bigcup_{x\in K} B_{r_x-\delta}(x)\right)^c$. Hence, to estimate their number it is sufficient to estimate the volume of this region.

Now, consider the quantity
\[
h(t):=\mathcal{L}^n\left(\bigcup_{x\in K} B_{r_x+t}(x)\right).
\]
The volume we want to estimate is given by $h(\varepsilon+\delta)-h(-\delta)$. To study the function $h$, observe that its derivative is given by the perimeter of the same union of balls:
\[
h'(t)=\HH^{n-1}\left(\partial\left(\bigcup_{x\in K} B_{r_x+t}(x)\right)\right).
\]
Since $h(\delta+\varepsilon)-h(-\delta)=(\varepsilon+2\delta) h'(s)$ for some $s\in(-\delta,\delta)$, and since Proposition~\ref{prop_loc1Gamloc_A3_0a} gives a bound on $h'$ only depending on $R$ and $\diam A$, we may estimate the required volume by $C(\varepsilon+2\delta)$. This implies that the number of disjoint cubes of side $\delta/\sqrt{n}$ completely contained in such a region is bounded above by $C(\varepsilon+\delta)/\delta^{n}$, which concludes the proof.
\end{proof}

\begin{corollary}\label{lm_loc1Gamloc_A3}
Let $A:=\left(\bigcup_{x\in K}B_{r_x}(x)\right)^c$, where $K\subset\R^n$ is a finite set and $r_x\ge R$ for some $R>0$ and for all $x\in K$. Then there is a $\delta$-net $\mathcal{G}_{Surf}(A,\delta)\subset\partial A$ of the boundary $\partial A$ such that
\[
\# \mathcal{G}_{Surf}(A,\delta) \leq \frac{C}{\delta^{n-1}},
\]
for some $C>0$ depending only on $R$ and on $\diam A$ (with the continuous dependence on these parameters).
\end{corollary}

\begin{proof}
It is sufficient to set $\mathcal{G}_{Surf}(A,\delta):=\mathcal{G}_{Vol}(A,0,\delta)$ and apply the previous Lemma~\ref{lm_loc1Gamloc_A4}.
%
%
\end{proof}
\bibliographystyle{plain}

\end{document}